\documentclass[11pt]{amsart}
\usepackage[letterpaper, total={6.5in,9in}, includefoot, centering]{geometry}
\usepackage{amsfonts}
\usepackage{amsmath,amsthm,amssymb,amscd}
\usepackage[foot]{amsaddr}

\usepackage[mathscr]{eucal}
\usepackage{latexsym}
\usepackage[new]{old-arrows}
\usepackage{graphics}
\usepackage{verbatim}
\usepackage{upgreek}
\usepackage{overpic}
\usepackage[all,cmtip]{xy} 
\usepackage{enumitem}
\usepackage{tikz}
\usepackage{tikz-cd}
\usepackage[pagebackref]{hyperref}
\renewcommand*\backref[1]{\ifx#1\relax \else (Page #1) \fi}
\usepackage[font=scriptsize]{caption}
\usepackage{upgreek}
\makeatletter
\newsavebox{\@brx}
\newcommand{\llangle}[1][]{\savebox{\@brx}{\(\m@th{#1\langle}\)}%
  \mathopen{\copy\@brx\kern-0.5\wd\@brx\usebox{\@brx}}}
\newcommand{\rrangle}[1][]{\savebox{\@brx}{\(\m@th{#1\rangle}\)}%
  \mathclose{\copy\@brx\kern-0.5\wd\@brx\usebox{\@brx}}}
\makeatother
\graphicspath{ {./images/} }
\usepackage{tikz}
\usepackage{tikz-cd}
\usepackage{bbm}

\usepackage[authoryear,sort&compress]{natbib}
\usepackage{adjustbox}
\bibpunct{[}{]}{;}{n}{,}{,}

\makeatletter \@addtoreset{equation}{section}

\makeatother
\newtheorem{theorem}{Theorem}[section]

\newtheorem{definition}{Definition}[section]
\newtheorem{remark}{Remark}[section]
\newtheorem{example}{Example}[section]
\newtheorem{prop}{Proposition}[section]

\newtheorem*{prop*}{Proposition}

\usepackage{xcolor}

\begin{document}
\title{On Properties of Adjoint Systems for Evolutionary PDEs}
\author{Brian K. Tran$^\dagger$}
\author{Ben S. Southworth$^\dagger$}
\author{Melvin Leok$^\ddagger$}
\address{$^\dagger$Los Alamos National Laboratory, Theoretical Division, Los Alamos, New Mexico 87545}
\address{$^\ddagger$UC San Diego, Department of Mathematics, La Jolla, CA 92093}
\email{btran@lanl.gov, southworth@lanl.gov, mleok@ucsd.edu}
\allowdisplaybreaks

\begin{abstract}
We investigate the geometric structure of adjoint systems associated with evolutionary partial differential equations at the fully continuous, semi-discrete, and fully discrete levels and the relations between these levels. We show that the adjoint system associated with an evolutionary partial differential equation has an infinite-dimensional Hamiltonian structure, which is useful for connecting the fully continuous, semi-discrete, and fully discrete levels. We subsequently address the question of discretize-then-optimize versus optimize-then-discrete for both semi-discretization and time integration, by characterizing the commutativity of discretize-then-optimize methods versus optimize-then-discretize methods uniquely in terms of an adjoint-variational quadratic conservation law. For Galerkin semi-discretizations and one-step time integration methods in particular, we explicitly construct these commuting methods by using structure-preserving discretization techniques. 
\end{abstract}

\maketitle

\tableofcontents

\section{Introduction}
In this paper, we investigate adjoint systems associated with evolution equations on infinite-dimensional Banach spaces, at the fully continuous, semi-discrete, and fully discrete levels, with the aim of addressing theoretical and practical questions in the optimization and optimal control of evolutionary partial differential equations (PDEs). 

The solution of many nonlinear optimization and optimal control problems involves successive linearization, and as such variational equations and their adjoints play a critical role in a variety of applications. Adjoint equations are of particular interest when the parameter space is of significantly higher dimension than that of the output or objective. In particular, the simulation of adjoint equations arise in sensitivity analysis~\cite{Ca1981, CaLiPeSe2003}, adaptive mesh refinement~\cite{LiPe2003}, uncertainty quantification~\cite{WaDuAlIa2012}, automatic differentiation~\cite{Gr2003}, superconvergent functional recovery~\cite{PiGi2000}, optimal control~\cite{Ro2005, CaObTr2014}, optimal design~\cite{GiPi2000}, optimal estimation~\cite{NgGeBe2016}, and deep learning viewed as an optimal control problem~\cite{DeCeEhOwSc2019}.

The study of geometric aspects of adjoint systems arose from the observation that the combination of any system of differential equations and its adjoint equations are described by a formal Lagrangian~\cite{Ib2006, Ib2007}. This naturally leads to the question of when the formation of adjoints and discretization commutes~\cite{SoTz1997}, and prior work on this include the Ross--Fahroo lemma~\cite{RoFa2001}, and the observation by \citet{Sa2016} that the adjoints and discretization commute if and only if the discretization is symplectic, in the specific setting of Runge--Kutta methods. Recently, in \cite{TrLe2024}, we investigate the symplectic and presymplectic structures for adjoint systems associated with ordinary differential equations (ODEs) and differential-algebraic equations (DAEs), respectively, and show that the processes of adjoining, discretization, and reduction (for index 1 DAE) commute with appropriate choices of these processes. 

In this paper, we extend these previous studies of discretizing adjoint systems by considering evolutionary PDEs and their associated adjoint systems, at the fully continuous, semi-discrete (i.e. discretized in space), and fully discrete levels, and investigate the connections between them. In particular, we utilize techniques from symplectic geometry to provide a precise geometric characterization of when semi-discretization, time integration, and adjoining commute.

\subsection{Discretize-then-Optimize versus Optimize-then-Discretize}\label{section:DtO-vs-OtD}
There is a vast literature on solving optimization and optimal control problems for differential equations via adjoint methods; we will not provide an exhaustive list but refer the reader to the following references and the references therein. Studies of adjoint systems for ODEs and DAEs are performed in \cite{DoZhSaSa2014, CaLiPeSe2003, TrLe2024, Sa2016}, in the context of optimal control in \cite{Ro2005, EiLa2020, Wa2007, CaObTr2014}, and in the context of neural networks in \cite{ChRuBeDu2018, GhKeBi2019, MaMiYa2023}. For PDEs, adjoint methods for time-dependent PDEs with adaptive mesh refinement is studied in \cite{LiPe2004}, adjoints for specific PDE systems are studied in \cite{NoWa2007, SaKa2000, KuCr2016, Pr2000}, and \textit{a posteriori} analysis of adjoint systems is reviewed in \cite{GiSu2002}.

One of the major conceptual questions is when to use discretize-then-optimize (DtO) methods, also known as direct methods, versus optimize-then-discretize (OtD) methods, also known as indirect methods, and in particular, what are the discrepancies between the two approaches. It is known that these two different approaches, in general, do not commute, and can lead to different discrete gradients of objective functions. 

For ODEs, it is known that DtO recovers the exact discrete gradients, i.e., DtO produces the exact gradient for the discrete-time optimization problem, whereas OtD may not (see, for example, \cite{GhKeBi2019, Sa2016}). This leads to the result that DtO methods produce gradients that are independent of the error of the method for the state variable \cite{OnRu2020}, and thus produce proper descent directions for discrete objective functions. On the other hand, the exact discrete gradient produced from DtO is not always \textit{adjoint consistent}, i.e., a consistent approximation of the continuous-time gradient, see, for example, \cite{AlSa2009}.

Extending the discussion to PDEs, one also must include the process of spatial discretization which can lead to further discrepancies in DtO and OtD \cite{LiPe2004, NoWa2007}. For Galerkin or Galerkin-like methods (e.g., discontinuous Galerkin methods), there is a similar issue of adjoint consistency for DtO methods \cite{Ha2007, Hi2013, ArBr2002}.

One of the main goals of this paper is to provide a geometric characterization of the discrepancy between DtO versus OtD methods, both for semi-discretization and time integration. 

\subsection{Main Results}
In this paper, we perform a systematic study of adjoint systems associated with evolution equations on infinite dimensional Banach spaces at the fully continuous, semi-discrete, and fully discrete levels, and of the relations amongst these levels. Vital to our analysis will be the construction of the infinite-dimensional Hamiltonian structure associated with adjoint systems for evolution equations, which will pave the way to a geometric characterization of the relations between these levels.

In Section \ref{section:adjoint-systems-evolution-section}, we investigate the fully continuous adjoint system associated with semilinear evolution equations. First, in Section \ref{Semilinear Evolution Equation Section}, we recall some basic facts about semilinear evolution equations, their semigroups and associated adjoint semigroups, and in Section \ref{section:inf-dim-hamiltonian-systems} we review infinite-dimensional Hamiltonian systems. In Section \ref{section:adjoint-systems-evolution-pde}, we define adjoint systems associated with evolution equations and equip such systems with an infinite-dimensional Hamiltonian structure. We discuss existence and uniqueness of solutions for these systems with Type II boundary conditions, specifying initial conditions for the state variable and terminal conditions for the adjoint variable and we further show that the adjoint system obeys an adjoint-variational quadratic conservation law given in Proposition \ref{prop:adjoint-cons-semilinear}, which is fundamental to adjoint sensitivity analysis. 

In Section \ref{section:discretization}, we explore the discretization of adjoint systems for evolutionary PDEs, both at the semi-discrete and fully discrete levels. In Section \ref{section:semi-discretization}, we study \textit{method-of-lines} semi-discretizations of adjoint systems. In particular, we show that associated with a method-of-lines semi-discretization of the evolution equation, there is an associated dual semi-discretization of the adjoint system such that semi-discretization and adjoining commute (Theorems \ref{thm:semi-disc-galerkin-comm} and \ref{thm:method-of-lines-adjoint-comm}). We show how this implies that in order for semi-discretize-then-optimize and optimize-then-semi-discretize to commute, the optimize-then-semi-discretize method must preserve the Hamiltonian structure. We further characterize the associated dual semi-discretization as the unique semi-discretization of the adjoint system which satisfies a semi-discrete analogue of the adjoint-variational quadratic conservation law in Theorem \ref{thm:method-of-lines-adjoint-comm}.

In Section \ref{section:time-integration}, we turn our attention to time integration of adjoint systems associated with ODEs via one-step methods. In particular, we show that time integration via one-step methods and adjoining commute precisely when the one-step method for the adjoint system is the cotangent lift of the one-step method for the state dynamics (Theorem \ref{thm:time-int-adjoint-comm}). This generalizes the result of \cite{Sa2016}, where it is shown that time integration via partitioned Runge--Kutta methods and adjoining commute if and only if the method is symplectic; namely, we generalize this result to the class of all one-step methods. In Theorem \ref{thm:time-int-adjoint-comm}, we additionally characterize the cotangent lift of a one-step method as the unique one-step method which covers the one-step method for the state variable and satisfies the discrete adjoint-variational quadratic conservation law. This addresses the question of DtO versus OtD for time integration via one-step methods: an OtD scheme which does not use the cotangent lifted method cannot satisfy the adjoint quadratic conservation law and hence, can lead to the observed defect in the exact discrete gradient of such methods as discussed in Section \ref{section:DtO-vs-OtD}. We furthermore show in Proposition \ref{prop:cotangent-order} that the cotangent lifted method is adjoint consistent provided that the underlying one-step method is \textit{variationally equivariant}, which informally is the property that taking variations and discretizing by the one-step method commute.

Finally, in Section \ref{section:naturality}, we combine the previous results for semi-discretization and time integration to discuss the natural relations between the fully continuous, semi-discrete, and fully discrete levels of the adjoint system associated with evolution equations. 

\section{Adjoint Systems for Evolutionary PDEs}\label{section:adjoint-systems-evolution-section}
In this section, we investigate adjoint systems associated with semilinear evolution equations and in particular, utilize techniques from infinite-dimensional symplectic geometry \cite{ChMa1974} to equip such systems with a symplectic structure. This structure will be useful for our investigation of semi-discretization and time integration of these systems. It should be noted that much of our discussion can be applied, at least formally, to more general nonlinear evolution equations, although analytic issues such as existence and uniqueness would be problem dependent. Throughout the paper, we will assume that $X$ is a real reflexive Banach space, $\langle\cdot,\cdot\rangle$ will denote the duality pairing between $X^*$ and $X$, and $B^*$ will denote the adjoint of an operator $B$ with respect to this duality pairing. Note that the assumption that $X$ is reflexive is crucial, since we require that the adjoint semigroup of a strongly continuous semigroup is also strongly continuous; this is necessary for the existence and uniqueness of the Hamiltonian vector field of an adjoint system in the infinite-dimensional setting. 

We start by recalling some basic facts about adjoint systems in the finite-dimensional (ODE) setting.

\subsection{Geometry of Finite-Dimensional Adjoint Systems}\label{section:geometry-adjoints-review}
In \cite{TrLe2024}, we provide a systematic study of the geometry of adjoint systems associated with ODEs and DAEs, utilizing tools from symplectic and presymplectic geometry, respectively. This paper extends the results of \cite{TrLe2024} to the infinite-dimensional setting, where we will consider adjoint systems associated with evolution equations on a Banach space. As a primer, it will be useful to recall some facts on the geometry of adjoint systems in the finite-dimensional setting (for more details, see \cite{TrLe2024}).

Let $M$ be a finite-dimensional manifold and consider the ODE on $M$ given by
\begin{equation}\label{ODEmanifold}
\dot{q} = f(q),
\end{equation} 
where $f$ is a vector field on $M$. Letting $\pi: TM \rightarrow M$ denote the tangent bundle projection, we recall that a vector field $f$ is a map $f: M \rightarrow TM$ which satisfies $\pi \circ f = \mathbf{1}_M$, i.e., $f$ is a section of the tangent bundle.

Consider the Hamiltonian $H: T^*M \rightarrow \mathbb{R}$ given by $H(q,p) = \langle p, f(q) \rangle_q$ where $\langle\cdot,\cdot\rangle_q$ is the duality pairing of $T^*_qM$ with $T_qM$. Recall that the cotangent bundle $T^*M$ possesses a canonical symplectic form $\Omega = -d\Theta$ where $\Theta$ is the tautological one-form on $T^*M$. With coordinates $(q,p) = (q^A, p_A)$ on $T^*M$, this symplectic form has the coordinate expression $\Omega = dq\wedge dp \equiv dq^A \wedge dp_A$. 

We define the adjoint system as the ODE on $T^*M$ given by Hamilton's equations, with the above choice of Hamiltonian $H$ and the canonical symplectic form. Thus, the adjoint system is given by the equation
$$ i_{X_H}\Omega = dH, $$
whose solution curves on $T^*M$ are the integral curves of the Hamiltonian vector field $X_H$. As is well-known, for the particular choice of Hamiltonian $H(q,p) = \langle p, f(q)\rangle_q$, the Hamiltonian vector field $X_H$ is given by the cotangent lift $\widehat{f}$ of $f$, which is a vector field on $T^*M$ that covers $f$ (for a discussion of the geometry of the cotangent bundle and lifts, see \cite{YaIs1973, MaRa1999}). This is stated in the following proposition.
\begin{prop}
    For the Hamiltonian $H(q,p) = \langle p, f(q)\rangle_q$, the corresponding Hamiltonian vector field $X_H$ is equal to the cotangent lift $\widehat{f}$ of $f$.
\end{prop}
\begin{proof}
    To be more explicit, recall that the cotangent lift of $f$ is constructed as follows. Let $\Phi_{\epsilon}: M \rightarrow M$ denote the time-$\epsilon$ flow map of $f$, which generates a one-parameter family of diffeomorphisms. Then, we consider the cotangent lifted diffeomorphisms given by $T^*\Phi_{-\epsilon}: T^*M \rightarrow T^*M$. This covers $\Phi_{\epsilon}$ in the sense that $\pi_{T^*M} \circ T^*\Phi_{-\epsilon} = \Phi_{\epsilon} \circ \pi_{T^*M} $ where $\pi_{T^*M}: T^*M \rightarrow M$ is the cotangent bundle projection. The cotangent lift $\widehat{f}$ is then defined to be the infinitesimal generator of the cotangent lifted flow,
$$ \widehat{f}(z) = \frac{d}{d\epsilon}\Big|_{\epsilon=0} T^*\Phi_{-\epsilon} (z). $$
We can directly verify that $\widehat{f}$ is the Hamiltonian vector field for $H$, which follows from
$$ i_{\widehat{f}}\Omega = -i_{\widehat{f}}d\Theta = -\mathcal{L}_{\widehat{f}}\Theta + d( i_{\widehat{f}}\Theta ) = d( i_{\widehat{f}}\Theta) = dH, $$
where $\mathcal{L}_{\hat{f}}\Theta = 0$ follows from the fact that cotangent lifted flows preserve the tautological one-form and $H = i_{\widehat{f}}\Theta$ follows from a direct computation, where $i_{\widehat{f}}\Theta$ is interpreted as a function on the cotangent bundle which maps $(q,p)$ to $\langle \Theta(q,p), \widehat{f}(q,p)\rangle_q$.
\end{proof}

With coordinates $z = (q,p)$ on $T^*M$, the adjoint system is the ODE on $T^*M$ given by
\begin{equation}\label{AdjointODEmanifold}
\dot{z} = \widehat{f}(z). 
\end{equation}

The adjoint system \eqref{AdjointODEmanifold} covers \eqref{ODEmanifold} in the following sense.
\begin{prop}[Proposition 2.2 of \cite{TrLe2024}] \label{LiftIntegralCurveProp}
Integral curves to the adjoint system \eqref{AdjointODEmanifold} lift integral curves to the system \eqref{ODEmanifold}.
\end{prop}

In coordinates, the adjoint system has the form
\begin{align*}
    \dot{q} &= f(q), \\
    \dot{p} &= -[Df(q)]^*p.
\end{align*}
We will often refer to the variable $q$ as the state variable and $\dot{q} = f(q)$ as the state dynamics; we refer to $p$ as the adjoint or costate variable and $\dot{p} = -[Df(q)]^*p$ as the adjoint equation. We refer to both equations together as the adjoint system.

The adjoint system possesses a quadratic invariant associated with the variational equations of \eqref{ODEmanifold}. The variational equation is given by considering the tangent lifted vector field on $TM$, $\widetilde{f}: TM \rightarrow TTM$, which is defined in terms of the flow $\Phi_{\epsilon}$ generated by $f$ by
$$ \widetilde{f}(q,\delta q) = \frac{d}{d\epsilon}\Big|_{\epsilon = 0} T\Phi_{\epsilon} (q,\delta q), $$
where $(q,\delta q)$ are coordinates on $TM$. That is, $\widetilde{f}$ is the infinitesimal generator of the tangent lifted flow. The variational equation associated with \eqref{ODEmanifold} is the ODE associated with the tangent lifted vector field. In coordinates, 
\begin{equation}\label{VariationalODEmanifold}
\frac{d}{dt}(q,\delta q) = \widetilde{f}(q,\delta q).
\end{equation}
\begin{prop}[Proposition 2.3 of \cite{TrLe2024}] \label{prop:ManifoldQuadraticInvariant}
For a curve $(q,\delta q,p)$ on $TM \oplus_M T^*M$ satisfying \eqref{AdjointODEmanifold} and  \eqref{VariationalODEmanifold},
\begin{equation}\label{ManifoldVariationalQuadratic}
\frac{d}{dt} \Big\langle p(t), \delta q(t)\Big\rangle_{q(t)} = 0. \end{equation}
\end{prop}
This quadratic invariant is the key to adjoint sensitivity analysis \cite{Sa2016} and its invariance is a consequence of symplecticity of the Hamiltonian flow \cite{TrLe2024}. 

In moving to the infinite-dimensional setting, we will develop the infinite-dimensional Hamiltonian structure of adjoint systems associated with evolution equations. We will further investigate the reduction of this structure under semi-discretization and its preservation under geometric time integration. As a result, this will allow us to geometrically characterize the discrepancy between DtO and OtD methods. 

\subsection{Semilinear Evolutionary PDEs}\label{Semilinear Evolution Equation Section}
We recall some facts about abstract semilinear initial value problems of the following form. Let $A: \mathfrak{D}(A) \subset X \rightarrow X$ be a closed densely defined unbounded linear operator, with domain $\mathfrak{D}(A)$. Fix $t_f>0$ and let $f: [0,t_f] \times X \rightarrow X$. We then consider the following initial value problem problem
\begin{subequations}\label{Semilinear Evolution System}
\begin{align}\label{Semilinear Evolution Equation}
\dot{q}(t) &= Aq(t) + f(t,q(t)), \\
q(0) &= q_0. \label{Semilinear Evolution Equation IC}
\end{align}
\end{subequations}
Note that we allow $f$ to be a time-dependent and generally nonlinear operator; this includes inhomogeneous source terms of the form $f(t)$, as well as time-independent nonlinearities of the form $f(q)$.

We will assume throughout that $A$ generates a $C_0$-semigroup $\{e^{tA}\}_{t \geq 0}$. This holds, for example, if $A$ is a maximally contractive operator, by the Hille--Yosida theorem \cite{Pa1983}. The $C_0$-semigroup is characterized by the properties
\begin{align*}
e^{0A} &= I_X, \\
e^{sA} e^{tA} &= e^{(s+t)A}, \text{ for all } s,t \geq 0, \\
 \lim_{t \rightarrow 0} \| e^{tA}x - x \|_X &= 0, \text{ for all } x \in X.
\end{align*}
The first two properties are the statement that $\{e^{tA}\}_{t\geq 0}$ generates a semigroup on $X$ and the third property is the statement that the semigroup is strongly continuous ($C_0$). The generator $A$ is related to its semigroup by
$$ Ax = \lim_{t \rightarrow 0} \frac{e^{tA}x-x}{t}, $$
where the domain $\mathfrak{D}(A)$ is the subspace of $X$ where the above limit exists. 

Given $q_0 \in \mathfrak{D}(A)$, the curve $q(t) = e^{tA}q_0$ is the unique solution to the evolution equation \eqref{Semilinear Evolution System} with $f \equiv 0$. For the evolution equation \eqref{Semilinear Evolution System} with generally nonzero $f$, there are various assumptions which one can impose on $f$ to ensure that the evolution equation admits a solution, at least locally \cite{Pa1983, Ne1992}. For our purposes, we will assume that $f: [0,t_f] \times X \rightarrow X$ is Lipschitz continuous in both variables. Then, for $q_0 \in \mathfrak{D}(A)$, \eqref{Semilinear Evolution System} admits a unique solution which is characterized by Duhamel's principle~\cite{Pa1983}
\begin{equation}\label{eq:q-representation-formula}
    q(t) = e^{tA} q_0 + \int_{0}^{t} e^{(t-s)A}f(s,q(s)) ds.
\end{equation}
The solution is strong in the sense that $q \in C^1([0,t_f],X) \cap C^0([0,t_f],\mathfrak{D}(A))$.

We will be interested in a dual problem associated with the evolution equation \eqref{Semilinear Evolution Equation}. To investigate this dual problem, we will need the notion of the adjoint of the unbounded operator $A$. We define the associated adjoint operator $A^*: \mathfrak{D}(A^*)\subset X^* \rightarrow X^*$ as follows: the domain $\mathfrak{D}(A^*)$ is the set of all $x^* \in X^*$ for which there exists a $y^* \in X^*$ such that
$$ \langle y^*,x\rangle = \langle x^*, Ax\rangle, \text{ for all } x \in \mathfrak{D}(A). $$
Since $\mathfrak{D}(A)$ is dense, if such a $y^*$ exists, it is unique, and hence, we define $A^*x^* = y^*$. Note that since $A$ is by assumption closed and densely defined, $A^*$ is weak$^*$-densely defined and weak$^*$-closed \cite{Pa1983, Ne1992}. Since, by assumption, $X$ is reflexive, $A^*$ is weak-densely defined and weakly closed. Thus, $A^*$ generates a weakly continuous semigroup on $X^*$, $\{e^{tA^*}\}_{t \geq 0}$, which is in fact the adjoint of the $C_0$ semigroup $e^{tA}$, i.e., $e^{tA^*} = (e^{tA})^*$. By the weak semigroup theorem, $\{e^{tA^*}\}_{t \geq 0}$ is in fact strongly continuous, i.e., it is a $C_0$ semigroup on $X^*$ \cite{Pa1983, Ne1992}.

\subsection{Infinite-dimensional Hamiltonian Systems}\label{section:inf-dim-hamiltonian-systems}
Consider the cotangent bundle of a real reflexive Banach space $X$, $\pi_{T^*X}: T^*X \rightarrow X$, where, using the natural identification $T^*X \cong X \times X^*$, the bundle projection is given by $\pi_{T^*X}(q,p) = q$. Define the canonical one-form on $T^*X$ as follows: for $(q,p) \in T^*X$ and $(v,w) \in T_{(q,p)}(T^*X) \cong X \times X^*$,
$$ \Theta(q,p) \cdot (v,w) := \langle p, T\pi_{T^*X} (v,w)\rangle = \langle p,v\rangle. $$
The canonical symplectic form on $T^*X$ is defined as $\Omega = -d\Theta$; for $(v_1,w_1),(v_2,w_2) \in T_{(q,p)}(T^*X)$, we have the expression
$$ \Omega(q,p)\cdot ((v_1,w_1),(v_2,w_2)) = \langle w_2,v_1\rangle - \langle w_1,v_2\rangle. $$

Let $\mathfrak{D}^1$ be a dense subspace of $X$ which is additionally a Banach space with respect to a norm $\|\cdot\|_{\mathfrak{D}^1}$ and let $\mathfrak{D}^2$ be a dense subspace of $X^*$ which is additionally a Banach space with respect to a norm $\|\cdot\|_{\mathfrak{D}^2}$. A Hamiltonian is a map $H: \mathfrak{D}^1 \times \mathfrak{D}^2 \rightarrow \mathbb{R}$ which is Fr\'{e}chet differentiable in both arguments. Then, a Hamiltonian system is specified by a triple $(T^*X, \Omega, H)$ and associated with a Hamiltonian system are Hamilton's equations
$$ i_{X_H}\Omega = dH. $$
Since $X$ is reflexive, $\Omega$ is strongly non-degenerate and hence, the Hamiltonian vector field $X_H: \mathfrak{D}^1 \times \mathfrak{D}^2 \rightarrow TT^*X$ exists and is uniquely defined (for a detailed discussion of infinite-dimensional Hamiltonian systems, see \cite{ChMa1974}). The Hamiltonian vector field has the expression
\begin{align*}
X_H: \mathfrak{D}^1 \times \mathfrak{D}^2 &\rightarrow TT^*X, \\
     (q,p) &\mapsto (D_pH(q,p), -D_qH(q,p)) \in T_{(q,p)}T^*X \cong X \times X^*.
\end{align*}

We say that a curve $(q,p)$ on $T^*X$ is a solution of Hamilton's equations if it is an integral curve of $X_H$. Equivalently, such a solution satisfies 
\begin{align*}
\dot{q} &= D_pH(q,p), \\
\dot{p} &= -D_qH(q,p).
\end{align*}
$X_H$ exists and is uniquely defined, but this does not necessarily imply the existence or uniqueness of its corresponding integral curves; so one has to proceed with caution. Such existence and uniqueness will depend on the properties of $X_H$ and thus $H$; for example, if $X_H$ is Lipschitz, there exists a unique solution \cite{FuncAnalysis}. Below, we investigate solutions for our particular interest of adjoint systems.

\textbf{Time-Dependent Hamiltonian Systems}
To treat time-dependent evolutionary PDEs, we extend the definition of a Hamiltonian system to the case where the Hamiltonian is time-dependent, i.e., $H: [0,t_f] \times \mathfrak{D}^1 \times \mathfrak{D}^2 \rightarrow \mathbb{R}$. Let $\tau_t: [0,t_f] \times X \times X^* \rightarrow X \times X^*$ denote the canonical projection onto the second and third factors; we identify $\Theta$ and $\Omega$ with their pullbacks by $\tau_t$. Then, in the time-dependent setting, Hamilton's equations read
$$ i_{X_H}(\Omega - dH \wedge dt) = 0, $$
where the time-dependent Hamiltonian vector field has the expression 
\begin{align*}
X_H: (0,t_f) \times \mathfrak{D}^1 \times \mathfrak{D}^2 &\rightarrow T((0,t_f)) \oplus T(T^*X), \\
     (t,q,p) &\mapsto \frac{\partial}{\partial t} + (D_pH(t,q,p), -D_qH(t,q,p)).
\end{align*}

In this setting, we say that a curve on $[0,t_f] \times T^*X$ is a solution of Hamilton's equations if it is an integral curve of $X_H$ covering the identity on $[0,t_f]$. Equivalently, such a solution curve is of the form $t \mapsto (t, q(t), p(t))$, where
\begin{subequations}\label{eq:inf-dim-Hamilton-equation}
\begin{align}
\dot{q}(t) &= D_pH(t,q,p), \label{eq:inf-dim-Hamilton-equation-a} \\
\dot{p}(t) &= -D_qH(t,q,p). \label{eq:inf-dim-Hamilton-equation-b}
\end{align}
\end{subequations}

\subsection{Adjoint Systems for Evolutionary PDEs}\label{section:adjoint-systems-evolution-pde}
In this section, we consider a particular class of infinite-dimensional Hamiltonian systems. Namely, we will explore the adjoint system associated with the evolutionary PDE \eqref{Semilinear Evolution Equation}. Let $\mathfrak{D}^1 = \mathfrak{D}(A) \subset X$ and $\mathfrak{D}^2 = \mathfrak{D}(A^*) \subset X^*$. 

Define the \textit{adjoint Hamiltonian} associated with the evolutionary PDE \eqref{Semilinear Evolution Equation} by
\begin{subequations}\label{eq:adjoint-Hamiltonian}
    \begin{align}
        &H: [0,t_f] \times \mathfrak{D}(A) \times \mathfrak{D}(A^*) \rightarrow \mathbb{R}, \label{eq:adjoint-Hamiltonian-a}\\
        &H(t,q,p) = \langle p, Aq + f(t,q) \rangle \label{eq:adjoint-Hamiltonian-b}
    \end{align}
\end{subequations}
We will assume that $f: [0,t_f] \times X \rightarrow X$ is Lipschitz in both arguments. Furthermore, we assume that $f$ is differentiable in its second argument and that 
$$ \sup_{q \in X: \|q\|_X \leq R} \|D_qf(t,q)\|_{X^*} < \infty, $$
uniformly in $t \in [0,t_f]$, for any $R>0$.
Compute
\begin{align*}
D_qH(t,q,p) &= A^*p + [D_qf(t,q)]^*p, \\
D_pH(t,q,p) &= Aq + f(t,q).
\end{align*}
Thus, Hamilton's equations \eqref{eq:inf-dim-Hamilton-equation} for the adjoint Hamiltonian are
\begin{subequations}\label{eq:adjoint-system-semilinear}
\begin{align}
\dot{q}(t) &= Aq(t) + f(t,q(t)) \label{eq:adjoint-system-semilinear-a}, \\
\dot{p}(t) &= -A^*p(t) - [D_qf(t,q(t))]^*p(t). \label{eq:adjoint-system-semilinear-b}
\end{align}
\end{subequations}
We refer to equations \eqref{eq:adjoint-system-semilinear} as the adjoint system associated with the evolutionary PDE \eqref{Semilinear Evolution Equation}. To solve the adjoint system, we have to specify appropriate boundary conditions for the problem. We will consider Type II boundary conditions $q(0) = q_0 \in \mathfrak{D}(A)$, $p(t_f) = p_f \in \mathfrak{D}(A^*)$, which specifies an initial condition for the state variable $q$ and a terminal condition for the adjoint variable $p$. As motivation for these boundary conditions, let us consider as an example the linear case, $f \equiv 0$.

\begin{example}[Linear Case]\label{ex:linear-adjoint-system}
    In the linear case $f \equiv 0$, the semilinear evolution PDE is simply $\dot{q} = Aq$; the associated adjoint system is 
    \begin{align*}
        \dot{q} &= Aq, \\
        \dot{p} &= -A^*p.
    \end{align*}
    We would like to use the semigroups $\{e^{tA}\}_{t \geq 0}$ and $\{e^{tA^*}\}_{t \geq 0}$ generated by $A$ and $A^*$, respectively, to construct solutions of the adjoint system. First note that we cannot, in general, use the semigroups to solve the adjoint system as an initial value problem, $(q(0),p(0)) = (q_0,p_0) \in \mathfrak{D}(A) \times \mathfrak{D}(A^*)$, due to the minus sign appearing in the second equation, i.e., the relevant operator to consider is $-A^*$, which may not generate a semigroup, without further assumptions \cite{Ne1992}. However, $-A^*$ does generate a semigroup in reverse time, $\{e^{(t_f-t)A^*}\}_{t \leq t_f}$. To see this, define a reverse time parameter $s := t_f-t$. Then, the equation $\dot{p} = -A^*p$ with condition $p(t_f) = p_f$ can be equivalently expressed as
\begin{align*}
\frac{d}{ds} p &= A^*p, \\
p\Big|_{s=0} &= p_f.
\end{align*}
Thus, this equation can be solved via the semigroup $\{e^{sA^*}\}_{s \geq 0}$ in the reverse time variable $s$, which is equivalent to using the semigroup $\{e^{(t_f-t)A^*}\}_{t \leq t_f}$ in the standard time variable $t$.

We can then solve the adjoint system if we place Type II boundary conditions $(q(0), p(t_f)) = (q_0, p_f) \in \mathfrak{D}(A) \times \mathfrak{D}(A^*)$ with $t_f > 0$. Thus, we can define a solution of Hamilton's equations with Type II boundary conditions by
\begin{align*}
(q,p): [0,t_f] &\rightarrow X \times X^*, \\
       (q(t), p(t)) &= (e^{tA}q_0, e^{(t_f-t)A^*}p_f).
\end{align*}
By the properties of the semigroups discussed previously, we have that this curve is a solution of Hamilton's equations with Type II boundary conditions, i.e.,
\begin{align*}
\dot{q}(t) &= Aq \text{ for all } t \in (0,t_f), \\
\dot{p}(t) &= -A^*p \text{ for all } t \in (0,t_f), \\
q(0) &= q_0 \in \mathfrak{D}(A), \\
p(t_f) &= p_f \in \mathfrak{D}(A^*),
\end{align*}
where the last two equations are interpreted in the strong $C_0$ sense as $t \rightarrow 0$ and $t \rightarrow t_f$, respectively. These are the natural boundary conditions to consider for adjoint sensitivity analysis, where the adjoint equation for $p$ is interpreted as evolving backward or ``backpropagating" in time.
\end{example}

Returning to the general case, we consider the adjoint system \eqref{eq:adjoint-system-semilinear} with boundary conditions $q(0) = q_0 \in \mathfrak{D}(A)$, $p(t_f) = p_f \in \mathfrak{D}(A^*)$. Note that the first equation \eqref{eq:adjoint-system-semilinear-a} admits a unique solution $q \in C^1([0,t_f],X) \cap C^0([0,t_f], \mathfrak{D}(A^*))$ with $q(0)=q_0$, as discussed in \eqref{Semilinear Evolution Equation Section}. In particular, $\sup_{t \in [0,t_f]}\|q(t)\|_X < \infty$. Thus, with this solution curve $q(t)$ fixed, we have that the map $(t,p) \mapsto -[D_qf(t,q(t))]^*p$
is Lipschitz in both arguments. Hence, by the theory discussed in Section \ref{Semilinear Evolution Equation Section}, there exists a unique solution $p \in C^1([0,t_f],X^*) \cap C^0([0,t_f], \mathfrak{D}(A^*))$ to the second Hamilton's equation above with $p(t_f)=p_f$.

Now, we show that the adjoint system admits an \textit{adjoint-variational} quadratic invariant. To do this, we define the variational equation as the linearization of $\dot{q}(t) = Aq(t) + f(t,q(t))$ about the solution curve $q(t)$, i.e.,
\begin{align*}
\dot{\delta}q(t) &= A\delta q(t) + D_qf(t,q(t)) \delta q(t), \\
\delta q(0) &= \delta q_0 \in \mathfrak{D}(A).
\end{align*}
Again, by the assumptions on $f$, there exists a unique solution $\delta q \in C^1([0,t_f],X) \cap C^0([0,t_f], \mathfrak{D}(A))$ to the above variational equation. We can now state an adjoint sensitivity analysis result.
\begin{prop}\label{prop:adjoint-cons-semilinear}
Let the solution curves $q$, $p$, and $\delta q$ be as above. Then, for any $t \in (0,t_f)$,
\begin{equation}\label{eq:adjoint-cons-semilinear}
    \frac{d}{dt} \langle p(t), \delta q(t)\rangle = 0.
\end{equation}
\begin{proof}
Since $p \in C^1([0,t_f],X^*) \cap C^0([0,t_f], \mathfrak{D}(A^*))$ and $\delta q \in C^1([0,t_f],X) \cap C^0([0,t_f], \mathfrak{D}(A))$, we can directly compute
\begin{align*}
\frac{d}{dt} \langle p(t), \delta q(t)\rangle &= \langle \dot{p}(t), \delta q(t)\rangle + \langle p(t), \dot{\delta}q(t)\rangle \\
&= \langle -A^*p(t) - [D_qf(t,q(t))]^*p(t), \delta q(t)\rangle + \langle p(t), A\delta q(t) + D_qf(t,q(t))\delta q(t) \rangle = 0.
\end{align*}
\end{proof}
\end{prop}

\begin{remark}
    The quadratic invariant in Proposition \ref{prop:adjoint-cons-semilinear} is fundamental to adjoint sensitivity analysis for evolutionary PDEs. Although the main purpose of this paper is to investigate the structure of adjoint systems and their discretization, we will briefly outline the concept of adjoint sensitivity analysis here. 

    Formally, for a minimization problem
    $$ \min_{q_0 \in X} \text{C}\,(q(t_f)) \text{ such that } \dot{q}(t) = Aq(t) + f(t,q(t)),\ \dot{q}(0) = q_0, $$
    the gradient of the cost function with respect to $q_0$ in a direction $\delta q_0$ is given by $\langle D\text{C}\,(q(t_f)), \delta q(t_f)\rangle$, where $\delta q(t)$ solves the variational equation with $\delta q(0) = \delta q_0$. If we want to express this quantity in terms of the initial perturbation $\delta q_0$ instead of the perturbation propagated to the terminal time $\delta q(t_f)$, we can do this through the solution $p(t)$ of the adjoint equation by setting $p(t_f) = D\text{C}\,(q(t_f))$ and applying the previous proposition, which yields
    \begin{align*}
    \langle D\text{C}\,(q(t_f)), \delta q(t_f) \rangle = \langle p(t_f), \delta q(t_f)\rangle = \langle p(0), \delta q(0)\rangle.
    \end{align*}
    The gradient can then be computed by solving the variational equation forward in time or the adjoint equation backward in time. See for example \cite{Sa2016, Pr2000, GiSu2002, KuCr2016, SaKa2000, NoWa2007}.
\end{remark}

\section{Discretization of Adjoint Systems}\label{section:discretization}
In this section, we investigate the semi-discretization and time integration of the adjoint system associated with an evolutionary PDE. We show that associated with semi-discretization in space and time integration of an evolutionary PDE, there are naturally associated dual methods such that adjoining and discretization commute, and characterize these uniquely in terms of semi-discrete and fully discrete analogues of the adjoint-variational quadratic conservation law \eqref{eq:adjoint-cons-semilinear}.

\textbf{Notation.} For a linear operator $B$, we denote its adjoint with respect to a duality pairing $\langle\cdot,\cdot\rangle$ as $B^*$. As we will see, when we discuss discretization in the subsequent sections, several duality pairings arise. To avoid ambiguity, we introduce some extra notation. Let $V$ be a finite-dimensional vector space, $V^*$ its dual. For a duality pairing $\langle \cdot,\cdot\rangle_M: V^* \times V \rightarrow \mathbb{R}$, we will denote the adjoint of an operator $B$ with respect to this pairing as $B^{*M}$. When $V = \mathbb{R}^N$, we denote $\langle \cdot,\cdot\rangle_S$ as the standard duality pairing, and thus, the adjoint with respect to the standard duality pairing is just given by the transpose, $B^{*S} = B^T$. The use of several duality pairings will be fundamental in our work, as we will relate the adjoint systems formed from different pairings. The standard duality pairing and the mass matrix pairing are particularly common cases of duality pairings in the context of semi-discretization.

We will denote coordinates on a finite-dimensional vector space in bold, for example, $\mathbf{q} = (\mathbf{q}^i) \in V$ and $\mathbf{p} = (\mathbf{p}_i) \in V^*$, respectively. Consider two duality pairings $\langle \cdot, \cdot \rangle_{L}$ and $\langle \cdot,\cdot \rangle_{R}$ on $V^* \times V$, related by an invertible operator $P: V \rightarrow V$ such that
$$ \langle \mathbf{p}, \mathbf{v}\rangle_{L} = \langle \mathbf{p}, P \mathbf{v} \rangle_{R} \text{ for all } \mathbf{p} \in V^*, \mathbf{v} \in V.$$
Then, the adjoint of an operator $B$ with respect to  $\langle \cdot, \cdot \rangle_{L}$ is related to its adjoint with respect to $\langle \cdot, \cdot \rangle_{R}$ by
\begin{align}\label{eq:matrix-duality-similarity}
    B^{*L} = P^{-*R} B^{*R} P^{*R},
\end{align}
where $P^{-*R} := (P^{-1})^{*R} = (P^{*R})^{-1}$. Note in particular that $P^{*L} = P^{*R}$.

Now, consider an ODE on $V$ of the form $\dot{\mathbf{q}} = K\mathbf{q} + \mathbf{f}(t,\mathbf{q})$ where $K$ is a linear operator (of course, by letting $K=0$, this encompasses a fully nonlinear ODE as well), and let $\langle\cdot,\cdot\rangle_M$ be a duality pairing on $V \times V^*$. Then, we define the adjoint Hamiltonian associated with the pairing $\langle\cdot,\cdot\rangle_M$ as
$$ H^M(t,\mathbf{q},\mathbf{p}) = \langle \mathbf{p}, K\mathbf{q} + \mathbf{f}(t,\mathbf{q})\rangle_M. $$
The adjoint system \textit{induced} by the pairing $\langle\cdot,\cdot\rangle_M$ is given by Hamilton's equations
\begin{subequations}\label{eq:adjoint-system-M}
\begin{align}
    \dot{\mathbf{q}} &= \frac{\delta}{\delta \mathbf{p}} H^M(t,\mathbf{q},\mathbf{p}) = K\mathbf{q} + \mathbf{f}(t,\mathbf{q}), \label{eq:adjoint-system-M-a} \\
    \dot{\mathbf{p}} &= \frac{\delta}{\delta \mathbf{q}} H^M(t,\mathbf{q},\mathbf{p}) = -K^{*M} \mathbf{p} - [D_{\mathbf{q}}\mathbf{f}(t,\mathbf{q})]^{*M}\mathbf{p}, \label{eq:adjoint-system-M-b}
\end{align}
\end{subequations}
where the above directional (variational) derivatives of a real-valued function $G$ of $\mathbf{q}$ and $\mathbf{p}$ are defined by
\begin{align*}
    \left\langle \delta \mathbf{p}, \frac{\delta G}{\delta \mathbf{q}} \right\rangle_M &= \frac{d}{d\epsilon}\Big|_{\epsilon = 0} G(\mathbf{q}, \mathbf{p}+\epsilon \delta \mathbf{p}) \text{ for all } \delta \mathbf{p} \in V^*, \\
    \left\langle \frac{\delta G}{\delta \mathbf{q}}, \delta \mathbf{q} \right\rangle_M &= \frac{d}{d\epsilon}\Big|_{\epsilon = 0} G(\mathbf{q} + \epsilon \delta \mathbf{q}, \mathbf{p}) \text{ for all } \delta \mathbf{q} \in V.
\end{align*}

Throughout, we adopt the Einstein summation convention, where repeated upper and lower indices are summed over; for example, for $\{\mathbf{p}_i\}_{i=1}^s$ and $\{\mathbf{q}^i\}_{i=1}^s$,
$$ \mathbf{p}_i\mathbf{q}^i := \sum_{i=1}^s \mathbf{p}_i \mathbf{q}^i. $$

\subsection{Semi-Discretization of Adjoint Systems}\label{section:semi-discretization}
We now introduce the notion of a Galerkin semi-discretization of the evolution equation \eqref{Semilinear Evolution Equation} on $X$ into a finite-dimensional subspace $X_h$. We will subsequently drop the subspace requirement to allow for more general semi-discretizations.

As before, let $X$ be a reflexive Banach space with duality pairing $\langle\cdot,\cdot\rangle$ on $X^* \times X$. Let $X_h$ be a finite-dimensional subspace of $X$ with the inclusion $i_h: X_h \hookrightarrow X$, and let $\{\varphi_i\}_{i=1}^{\dim(X_h)}$ be a basis for $X_h$. Let $\{l^j\}_{j=1}^{\dim(X_h)}$ be a set of degrees of freedom for $X_h$, i.e., a basis for $X_h^*$. Let $\langle\cdot,\cdot\rangle_h$ denote the duality pairing on $X_h^* \times X_h$. A \textit{Galerkin semi-discretization} of \eqref{Semilinear Evolution Equation} is specified by a projection $\Pi_h: X \rightarrow X_h$ and an approximation $q(t) \approx \sum_i \mathbf{q}^i(t) \varphi_i $ satisfying the evolution equation on the degrees of freedom in the following sense:
\begin{equation}\label{eq:galerkin-semi-discretization-dof}
    \left\langle l^j, \dot{\mathbf{q}}^i(t) \varphi_i - \mathbf{q}^i(t) \Pi_h A i_h\varphi_i - \Pi_h f(t,\mathbf{q}^k(t)i_h\varphi_k) \right\rangle_h = 0,\ j=1,\dots,\dim(X_h).
\end{equation}
Note that equation \eqref{eq:galerkin-semi-discretization-dof} can be interpreted as the ``restriction" of the state dynamics to the degrees of freedom in the following sense: letting $q(t) = \sum_i \mathbf{q}^i(t) i_h\varphi_i \in X$ denote the semi-discrete curve, viewed as an element of $X$ through the inclusion, we have
\begin{align*}
    0 &= \left\langle l^j, \dot{\mathbf{q}}^i(t) \varphi_i - \mathbf{q}^i(t) \Pi_h A i_h\varphi_i - \Pi_h f(t,\mathbf{q}^k(t)i_h\varphi_k) \right\rangle_h \\
    &= \left\langle l^j, \dot{\mathbf{q}}^i(t) \Pi_h i_h \varphi_i - \mathbf{q}^i(t) \Pi_h A i_h\varphi_i - \Pi_h f(t,\mathbf{q}^k(t)i_h\varphi_k) \right\rangle_h \\
    &= \left\langle l^j, \Pi_h \left( \dot{\mathbf{q}}^i(t) i_h \varphi_i - \mathbf{q}^i(t)  A i_h\varphi_i - f(t,\mathbf{q}^k(t)i_h\varphi_k) \right) \right\rangle_h \\
    &= \left\langle \Pi_h^* l^j, \dot{\mathbf{q}}^i(t) i_h \varphi_i - \mathbf{q}^i(t)  A i_h\varphi_i - f(t,\mathbf{q}^k(t)i_h\varphi_k)  \right\rangle \\
    &= \langle \Pi_h^* l^j, \dot{q}(t) - Aq(t) - f(t,q(t))\rangle,
\end{align*}
where we used the fact that $\Pi_h i_h$ equals the identity on $X_h$, since $\Pi_h$ is a projection, and $\Pi_h^*: X_h^* \rightarrow X$ is the adjoint of $\Pi_h: X \rightarrow X_h$. This gives an interpretation of the semi-discretization \eqref{eq:galerkin-semi-discretization-dof} in terms of the duality pairing on $X^* \times X$; namely, that the quantity describing the state dynamics $\dot{q}(t) - Aq(t) - f(t,q(t))$ in $X$ vanishes when tested by functionals $\Pi_h^*l^j \in X^*$. This interpretation can be equivalently thought of as arising from a complementary subspace to $X_h$ in $X$, namely,
$$ X = X_h \oplus \text{ker}(\Pi_h), $$
from which $X_h^*$ is isomorphic to the annihilator of this complementary subspace,
$$ X_h^* \cong \text{annil}(\text{ker}(\Pi_h)) := \{ p \in X^*: \langle p,v\rangle = 0 \text{ for all } v \in \text{ker}(\Pi_h)\} \subset X^*. $$
This is essential in the Galerkin semi-discretization construction of the adjoint equations in order to make sense of pairings between elements of $X_h^*$, such as the degrees of freedom, and elements of $X$.

Now, let $M$ and $K$ denote mass and stiffness matrices, respectively, with entries
\begin{align*}
    M^j_{\ i} &= \langle l^j, \varphi_i\rangle_h, \\
    K^j_{\ i} &= \langle l^j, \Pi_h A i_h\varphi_i\rangle_h. 
\end{align*}
Let $\mathbf{q}(t)$ denote the vector in $\mathbb{R}^{\dim(X_h)}$ with components $\mathbf{q}^i(t)$, and let the semi-discretized semilinear term be denoted by the vector $\mathbf{f}(t,\mathbf{q})$ with entries
$$ \mathbf{f}^j(t,\mathbf{q}) = \langle l^j, \Pi_h f(t,\mathbf{q}^ki_h\varphi_k) \rangle_h.$$
Then, the semi-discretization can be expressed as
\begin{equation}\label{Semi-discrete Evolution Equation}
    M\frac{d}{dt}\mathbf{q} = K\mathbf{q} + \mathbf{f}(t,\mathbf{q}).
\end{equation}

Now, we form the adjoint system \cite{TrLe2024} for the semi-discrete system \eqref{Semi-discrete Evolution Equation}. First, since $M$ is invertible, we express the above as a standard ODE,
$$ \frac{d}{dt}\mathbf{q} = M^{-1}K \mathbf{q} + M^{-1}\mathbf{f}(t,\mathbf{q}).$$
Note that the adjoint system depends on the duality pairing on $X_h^* \times X_h$ by equations \eqref{eq:adjoint-system-M}. There are two immediately obvious choices of duality pairing. First, since we identify $X_h \cong \mathbb{R}^N$, $N = \dim(X_H)$, where the identification is $\mathbf{q}^i \varphi_i \cong \mathbf{q}$, an obvious choice of duality pairing is just the standard duality pairing on $\mathbb{R}^N$, $\langle \cdot , \cdot \rangle_S$. The adjoint system induced by the standard duality pairing, via equations \eqref{eq:adjoint-system-M}, is given by
\begin{subequations}\label{Semi-discrete Adjoint Canonical System}
\begin{align}\label{Semi-discrete Adjoint Canonical System a}
    \frac{d}{dt}\mathbf{q} &= M^{-1}K \mathbf{q} + M^{-1}\mathbf{f}(t,\mathbf{q}),\\
    \frac{d}{dt}\mathbf{z} &= -K^TM^{-T}\mathbf{z} - [D_{\mathbf{q}}\mathbf{f}(t,\mathbf{q})]^T M^{-T}\mathbf{z}, \label{Semi-discrete Adjoint Canonical System b}
\end{align}
\end{subequations}
where $D_{\mathbf{q}}\mathbf{f}$ is the usual Jacobian of $\mathbf{f}$ with respect to the argument $\mathbf{q}$. Alternatively, we can consider the duality pairing on $\mathbb{R}^N$ naturally induced by the mass matrix, i.e., 
$$ \langle \mathbf{p},\mathbf{v}\rangle_M = \mathbf{p}^TM\mathbf{v}. $$
The adjoint system induced by the mass matrix is given by 
\begin{subequations}\label{eq:semi-discrete-adjoint-2}
\begin{align}
    \frac{d}{dt}\mathbf{q} &= M^{-1}K \mathbf{q} + M^{-1}\mathbf{f}(t,\mathbf{q}), \label{eq:semi-discrete-adjoint-2a}\\
    \frac{d}{dt}\mathbf{p} &= - (K^{*M}M^{-*M} + [D_{\mathbf{q}}\mathbf{f}(t,\mathbf{q})]^{*M} M^{-*M} ) \mathbf{p}. \label{eq:semi-discrete-adjoint-2b}
\end{align}
\end{subequations}
By equation \eqref{eq:matrix-duality-similarity}, the second equation can equivalently be written as
$$ M^T \frac{d}{dt} \mathbf{p} = -K^T \mathbf{p} - [D_{\mathbf{q}}\mathbf{f}(t,\mathbf{q})]^T \mathbf{p}, $$
where we recall that the matrix transpose $^T$ is identified with the adjoint with respect to the standard pairing $^{*S}$ under the above isomorphism $X_h \cong \mathbb{R}^{\dim(X_h)}$.

Note that the above adjoint systems were formed by first semi-discretizing the evolution equation and subsequently forming the adjoint system. We will now reverse this process: we will first form the adjoint system at the continuous level and subsequently semi-discretize. Recall the continuous adjoint system is given by
\begin{align*}
    \dot{q} &= Aq + f(t,q), \\
    \dot{p} &= -A^*p - [Df(t,q)]^*p.
\end{align*}
To semi-discretize this system, we discretize the $q$ variable as before, $q(t) \approx \sum_i \mathbf{q}^i(t) \varphi_i$ with degrees of freedom given by $\{l_j\}$. For $p \in X^*$, we semi-discretize by using the basis $\{l^j\}$ of $X_h^*$ via $p(t) \approx \sum_j \mathbf{p}_j(t) l^j$ and degrees of freedom given by $\{\varphi_i\}$. Furthermore, the projection of the Galerkin method for the adjoint variable is $i_h^*$, the adjoint of the inclusion for the state value, whereas the inclusion for the adjoint variable is $\Pi_h^*$, the adjoint of the projection for the state variable. Note that $i_h^*$ is indeed a projection $X^* \rightarrow X_h^*$, since $i_h^* \Pi_h^* = (\Pi_hi_h)^*$ is the identity on $X_h^*$. As we will see, this is a natural choice of semi-discretization for $p$ since the resulting system is equivalent to \eqref{Semi-discrete Adjoint Canonical System} and \eqref{eq:semi-discrete-adjoint-2}. Furthermore, we will see in Theorem \ref{thm:method-of-lines-adjoint-comm} that it is the unique semi-discretization of the adjoint system covering the base semi-discretization and satisfying a semi-discrete analogue of equation \eqref{eq:adjoint-cons-semilinear}.

\begin{prop}
With the above choice of semi-discretization for the continuous adjoint system, we have the semi-discrete adjoint system 
\begin{subequations}\label{Semi-discrete Adjoint System}
\begin{align}
    M \frac{d}{dt} \mathbf{q} &= K \mathbf{q} + \mathbf{f}(t,\mathbf{q}),\label{Semi-discrete Adjoint System a} \\
    M^T \frac{d}{dt} \mathbf{p} &= - K^T\mathbf{p} - [D_{\mathbf{q}}\mathbf{f}(t,\mathbf{q})]^T\mathbf{p}.\label{Semi-discrete Adjoint System b}
\end{align}
\end{subequations}
Note that this is equivalent to the mass-matrix-induced semi-discrete adjoint system in \eqref{eq:semi-discrete-adjoint-2}.
\begin{proof}
    The semi-discretization \eqref{Semi-discrete Adjoint System a} of the evolution equation is the same as before \eqref{Semi-discrete Evolution Equation}, so we only have to verify \eqref{Semi-discrete Adjoint System b}. The semi-discretization of the adjoint equation is given by
    $$ \langle \dot{\mathbf{p}}_j l^j,\varphi_i\rangle_h = \langle -i_h^*A^*\Pi_h^* \mathbf{p}_j l^j, \varphi_i\rangle_h - \langle i_h^*[D_qf(t,i_h\mathbf{q}^k\varphi_k)]^*\Pi_h^*\mathbf{p}_jl^j,\varphi_i\rangle_h,\ i=1,\dots,\dim(X_h). $$
We consider each term in the above equation. The first term is the $i^{th}$ component of the vector $M^T d\mathbf{p}/dt$. The second term can be expressed as
$\langle -i_h^*A^* \Pi_h^*\mathbf{p}_j l^j, \varphi_i\rangle_h = -\mathbf{p}_j \langle l^j,\Pi_h A i_h\varphi_i\rangle, $
which is the $i^{th}$ component of $-K^T\mathbf{p}$. To see that the third term corresponds to the $i^{th}$ component of $- [D_{\mathbf{q}}\mathbf{f}(t,\mathbf{q})]^T\mathbf{p}$, we explicitly compute the Jacobian
\begin{align*}
    [D_{\mathbf{q}}\mathbf{f}(t,\mathbf{q})]^i_{\ j} &= \frac{\partial}{\partial \mathbf{q}^j}\mathbf{f}^i(t,\mathbf{q}) = \frac{\partial}{\partial \mathbf{q}^j}\langle l^i, \Pi_h f(t, \mathbf{q}^ki_h\varphi_k)\rangle_h \\
    &= \langle l^i, \Pi_h D_qf(t,\mathbf{q}^k\varphi_k)i_h\varphi_j\rangle_h = \langle i_h^*[D_qf(t,\mathbf{q}^k\varphi_k)]^*\Pi_h^* l^i, \varphi_j\rangle_h.
\end{align*}
Thus, we see that the third term is the $i^{th}$ component of $- [D_{\mathbf{q}}\mathbf{f}(t,\mathbf{q})]^T\mathbf{p}$.

\end{proof}
\end{prop}

Let us formally denote the semi-discretization procedures on $X$ as $S_h$ and the dual semi-discretization on $X^*$ as $S_h^*$. For brevity, we denote the right hand side of the continuous and semi-discrete evolution equations as
\begin{align*}
    g &= Aq + f(t,q), \\
    \mathbf{g} &= K \mathbf{q} + \mathbf{f}(t,\mathbf{q}).
\end{align*}
We denote the procedures of forming the adjoints with respect to the standard duality pairing and the mass matrix induced duality pairing as Adjoint$_S$ and Adjoint$_M$, respectively. Then the preceding discussion can be summarized in the following result. 

\begin{theorem}\label{thm:semi-disc-galerkin-comm}
    The system \eqref{Semi-discrete Adjoint System} arising from forming the continuous adjoint equation and semi-discretizing is equivalent to the systems \eqref{Semi-discrete Adjoint Canonical System} and \eqref{eq:semi-discrete-adjoint-2} that arise from semi-discretizing the state dynamics and forming a discrete adjoint under the appropriate duality pairing and inner product. That is, semi-discretization and forming the adjoint commute, with the above choices of semi-discretization, once composed with the appropriate transformations, as summarized in the commutative diagram \eqref{diagram:semi-discrete-adjoint-comm}.

\begin{equation}\label{diagram:semi-discrete-adjoint-comm}
\begin{tikzcd}[column sep=14ex, row sep =10ex]
	\dot{q} = g \arrow[rr, "\text{Adjoint}"] \arrow[d, "S_h"] & {} & \begin{tabular}{c} $\dot{q} = g$ \\ $\dot{p} = -[D_qg]^*p$ \end{tabular} \arrow[d,"(S_h{,}S_h^*)"] \\ 
    \dot{\mathbf{q}} = M^{-1}\mathbf{g} \arrow[r, "\text{Adjoint}_S"]  \arrow[rr, "\text{Adjoint}_M", bend left = 20] & \begin{tabular}{c} $\dot{\mathbf{q}} = M^{-1} \mathbf{g}$ \\ $\dot{\mathbf{z}} = -[D_{\mathbf{q}}\mathbf{g}]^TM^{-T}\mathbf{z}$ \end{tabular} \arrow[r, "\mathbf{z} = M^T\mathbf{p}", shift left] &  \begin{tabular}{c} $\dot{\mathbf{q}} = M^{-1} \mathbf{g}$ \\ $M^T\dot{\mathbf{p}} = -[D_{\mathbf{q}}\mathbf{g}]^T\mathbf{p}$ \end{tabular} \arrow[l, shift left]
\end{tikzcd}
\end{equation}
\end{theorem}

Although these two systems are equivalent via the coordinate transformation $\mathbf{p} = M^{-T}\mathbf{z}$, we note that each represent a canonical Hamiltonian system on $T^*X_h$ with different coordinate representations and duality pairings. The system \eqref{Semi-discrete Adjoint Canonical System} can be interpreted as a Hamiltonian system on $T^*X_h \cong X_h \times X_h^*$ with duality pairing given by the standard duality pairing on $\mathbb{R}^n$, $\langle \mathbf{x},\mathbf{y}\rangle_S = \mathbf{x}^T\mathbf{y}$, canonical symplectic form given by $\Omega_h^{S} = \langle d\mathbf{q}\wedge d\mathbf{z}\rangle_S = d\mathbf{q}^{\ T} \wedge d\mathbf{z}$, and Hamiltonian
$$ H_h^{S}(t,\mathbf{q},\mathbf{z}) = \mathbf{z}^T M^{-1}K\mathbf{q} + \mathbf{z}^T M^{-1} \mathbf{f}(t,\mathbf{q}). $$
On the other hand, \eqref{eq:semi-discrete-adjoint-2} and \eqref{Semi-discrete Adjoint System} can be interpreted as a Hamiltonian system on $T^*X_h$ with duality pairing $\langle \mathbf{x},\mathbf{y}\rangle_M = \mathbf{x}^TM\mathbf{y}$, with symplectic form $\Omega^M_h = \langle d\mathbf{q}\wedge d\mathbf{p}\rangle_M = (Md\mathbf{q})^T \wedge d\mathbf{p}$, and Hamiltonian
$$ H^M_h(t,\mathbf{q},\mathbf{p}) = \mathbf{p}^{\ T} K\mathbf{q} + \mathbf{p}^T \mathbf{f}(t,\mathbf{q}).  $$
That \eqref{Semi-discrete Adjoint Canonical System} is equivalent to \eqref{Semi-discrete Adjoint System} can be expressed as the fact that the mapping $T^t_M: (t,\mathbf{q},\mathbf{z})\mapsto (t,\mathbf{q},M^{-*}\mathbf{z})$ pulls back the associated Cartan forms as
$$ (T^t_M)^*(\Omega^M_h - dH^M \wedge dt) = \Omega^{S}_h - dH^{S}\wedge dt. $$
In the autonomous case, i.e., where $f$ does not depend explicitly on $t$, this can be expressed as the fact that the map $T_M: (\mathbf{q},\mathbf{z})\mapsto (\mathbf{q},M^{-*}\mathbf{z})$ is a symplectomorphism
$$ (T_M)^*\Omega^M_h = \Omega_h^{S}, $$
and pulls back the Hamiltonian as
$$ (T_M)^*H^M_h = H_h^{S}.$$

From a finite element and discretization perspective, the formulation of \eqref{eq:semi-discrete-adjoint-2} and \eqref{Semi-discrete Adjoint System} is more natural, as the duality pairing on $T^*X_h$ is induced from the duality pairing on $T^*X$ through the inclusion and projection. In particular, with $p = \mathbf{p}_j \Pi_h^*l^j \in X^*$ and $q = \mathbf{q}^i i_h\varphi_i \in X$, we have
$$ \langle \mathbf{p},\mathbf{q}\rangle_M = \mathbf{p}_j M^j_{\ i} \mathbf{q}^i = \mathbf{p}_j \langle l^j, \varphi_i\rangle_h \mathbf{q}^i = \mathbf{p}_j \langle l^j, \Pi_h i_h\varphi_i\rangle_h \mathbf{q}^i = \langle \mathbf{p}_j\Pi^*_h l^j, \mathbf{q}^i_h\varphi_i \rangle  = \langle p,q\rangle. $$
From this observation, it is straightforward to verify that for the mapping $T_h: X_h \times X^*_h \rightarrow X \times X^*$ defined by $T_h(\mathbf{q},\mathbf{p}) = (\mathbf{q}^i i_h\varphi_i , \mathbf{p}_j\Pi_h^*l^j)$,
the semi-discrete Hamiltonian and semi-discrete symplectic structure are related to their infinite-dimensional counterparts as 
\begin{equation}\label{eq:semi-inf-relation}
(T_h)^*\Omega = \Omega_h^M,\quad (T_h)^*H = H_h^M, 
\end{equation}
where $(T_h)^*$ is the pullback of $T_h$, mapping forms on $X \times X^*$ to forms on $X_h \times X_h^*$. To see this, for any $\mathbf{q}^i \varphi_i \in X_h$ and $\mathbf{p}_j l^j \in X_h^*$, we have
\begin{align*}
    ((T_h)^*H)(\mathbf{q},\mathbf{p}) &= H \circ T_h(\mathbf{q},\mathbf{p}) = H(\mathbf{q}^i i_h\varphi_i, \mathbf{p}_j\Pi_h^*l^j) = \langle \mathbf{p}_j\Pi_h^*l^j, \mathbf{q}^iAi_h\varphi_i + f(\mathbf{q}^ii_h\varphi_i)\rangle \\ 
    &= \mathbf{p}_j \langle l^j, \Pi_h A i_h \varphi_i\rangle_h \mathbf{q}^i + \mathbf{p}_j \langle l^j, \Pi_h f(\mathbf{q}^i i_h\varphi_i)\rangle_h = \mathbf{p}^T K \mathbf{q} + \mathbf{p}^T \mathbf{f}(\mathbf{q}) = H^M_h(\mathbf{q},\mathbf{p}),
\end{align*}
and a similar computation holds to show that $(T_h)^*\Omega = \Omega_h^M$. Equation \eqref{eq:semi-inf-relation} is the statement that the semi-discrete Hamiltonian structure is the Galerkin restriction of the infinite-dimensional Hamiltonian structure.

\begin{remark}
    In the commutative diagram \eqref{diagram:semi-discrete-adjoint-comm}, we utilize two choices of duality pairings $\langle\cdot,\cdot\rangle_M$ and $\langle\cdot,\cdot\rangle_S$. We include the standard duality pairing as it is the usual duality pairing used to form adjoint systems. We include the mass matrix induced duality pairing since the OtD method given by the dual semi-discretization $(S_h, S_h^*)$ is naturally equivalent to a DtO method with respect to the mass matrix, whereas it is only equivalent to the DtO method with respect to the standard duality pairing once composed with the appropriate transformation $T^t_M$. Note also that a diagram analogous to \eqref{diagram:semi-discrete-adjoint-comm} holds with replacing the standard duality pairing with an arbitrary duality pairing.
\end{remark}

Note that both systems satisfy a semi-discrete adjoint-variational conservation law. 

\begin{prop}
    The mass matrix induced adjoint system \eqref{Semi-discrete Adjoint System} admits a quadratic adjoint-variational conservation law. For a solution $\mathbf{p}(t)$ of the adjoint equation \eqref{Semi-discrete Adjoint System b} and a solution $\delta \mathbf{q}(t)$ of the variational equation associated with \eqref{Semi-discrete Evolution Equation}, covering the same solution $\mathbf{q}(t)$ of \eqref{Semi-discrete Evolution Equation}, 
    \begin{equation}\label{eq:semi-disc-quad-cons-law}
    \frac{d}{dt} \langle \mathbf{p}(t), \delta \mathbf{q}(t)\rangle_M = 0.
    \end{equation}
    Similarly, the standard duality pairing induced adjoint system \eqref{Semi-discrete Adjoint Canonical System} admits a quadratic adjoint-variational conservation law. For a solution $\mathbf{z}(t)$ of the adjoint equation  \eqref{Semi-discrete Adjoint Canonical System b} and a solution $\delta \mathbf{q}(t)$ of the variational equation covering the same solution $\mathbf{q}(t)$ of the semi-discrete evolution equation, 
    \begin{equation}\label{eq-semi-disc-quad-cons-law-standard}
    \frac{d}{dt} \langle \mathbf{z}(t), \delta \mathbf{q}(t)\rangle_S = 0.     
    \end{equation}
    \begin{proof}
        These follow from Proposition \ref{prop:ManifoldQuadraticInvariant} in the standard adjoint theory for ODEs.
    \end{proof}
\end{prop}

As we will see shortly when discussing more general semi-discretizations, once the semi-discretization $S_h$ for the state variable is fixed, and a duality pairing is chosen, then the full semi-discretization on $X_h \times X_h^*$ is the unique semi-discretization of the fully continuous adjoint system such that the conservation law, \eqref{eq:semi-disc-quad-cons-law} or \eqref{eq-semi-disc-quad-cons-law-standard}, corresponding to the choice of duality pairing holds.

\textbf{More general semi-discretizations.}
Note that we made a particular choice of semi-discretization of the Banach space $X$. When $X$ is a function space, it can be thought of as a spatial discretization via a subspace method, such as the finite element method. Once this semi-discretization of $X$ is fixed, there is a natural dual choice of semi-discretization for the adjoint system on $X^* \times X$ such that it arises as the adjoint system of the semi-discretized evolution equation on $X$. Also, note that this notion of semi-discretization is fairly general. $X$ need not be an infinite-dimensional function space; it applies to sequence spaces such as $l^2(\mathbb{R})$ or even cases where $X$ is finite-dimensional (in which case, a semi-discretization of $X$ can be seen as a dimensional reduction to a ``lower order" space).
    
However, more general semi-discretizations are possible. For example, $X_h$ need not be chosen as a subspace but is more generally an approximating space that one can define an approximation of $A$ on. This allows more flexibility in the choice of semi-discretization, such as mixed methods and Discontinuous Galerkin methods. Furthermore, this allows for more freedom in treating the semilinear term, e.g., by quadrature. In essence, as before, once a semi-discretization of the evolution equation is fixed, there is a natural dual choice of semi-discretization of the adjoint system such that the corresponding diagram of adjoining and semi-discretization commute. Abstractly, such a semi-discretization and its dual semi-discretization reduce the adjoint Hamiltonian system on $T^*X$ to an adjoint Hamiltonian system on $T^*X_h$.

To be more precise, we introduce the following more general notion of semi-discretization: a \textit{method-of-lines semi-discretization} of an evolution equation on a Banach space $X$ is a procedure for mapping the evolution equation into an ODE on a finite-dimensional vector space $X_h$. Note that $X_h$ need not be a subspace of $X$. More precisely, we introduce the following definition.
\begin{definition}\label{def:method-of-lines}
    A method-of-lines semi-discretization for the class of semilinear evolution equations on $X$ into a finite-dimensional vector space $X_h$ is specified by a mapping $K_h[\cdot,\cdot]$ whose inputs are a closed and densely defined unbounded operator $A$ on $X$ and a time-dependent nonlinear operator $f$ on $X$ and whose output is a time-dependent vector field $K_h[A,f]$ on $X_h$.

    Associated with the continuous evolution equation of the form \eqref{Semilinear Evolution Equation}, the method-of-lines semi-discretization has the associated ODE on $X_h$ given by
    \begin{equation} \label{eq:method-of-lines-semi-discretization}
    \frac{d}{dt} \mathbf{q} = K_h[A,f](t,\mathbf{q}), 
    \end{equation}
\end{definition}

\begin{example}
    A Galerkin semi-discretization is a particular example of a method-of-lines semi-discretization, where $X_h$ is a finite-dimensional subspace of $X$ with a projection $\Pi_h: X \rightarrow X$. In this case, mapping $K_h[\cdot,\cdot]$ is explicitly
    $$    K_h[A,f](t,\mathbf{q}) = M^{-1}K\mathbf{q} + M^{-1}\mathbf{f}(t,\mathbf{q}), $$
    where $K$ and $\mathbf{f}$ are defined in terms of $A$ and $f$ as discussed in the previous section:
   \begin{align*}
       K^j_{\ i} &= \langle l^j, \Pi_h A i_h\varphi_i\rangle_h.\\
       \mathbf{f}^j(t,\mathbf{q}) &= \langle l^j, \Pi_h f(t,\mathbf{q}^ki_h\varphi_k) \rangle_h
   \end{align*}
\end{example}

Of course, one wants that the solution to the semi-discrete problem converges as $h \rightarrow 0$, in some sense, to a solution of the continuous problem but we will not discuss this here as it will depend generally on the choice of semi-discretization. We will assume that $K_h[A,f](t,\mathbf{q})$ is differentiable in $\mathbf{q}$ given that $f(t,q)$ is differentiable in $q$. We will assume that any method-of-lines semi-discretization admits a solution on the interval $[0,t_f]$ where the semilinearity $f$ and its derivative are uniformly bounded as discussed in Section \ref{section:adjoint-systems-evolution-pde} (which is the case for Galerkin semi-discretization, since the associated semi-discrete semilinear term $\mathbf{f}$ enjoys the same bounds).

Now, we can state the following very general result regarding adjoining and semi-discretization.

\begin{theorem}\label{thm:method-of-lines-adjoint-comm}
    A method-of-lines semi-discretization of the adjoint system \eqref{eq:adjoint-system-semilinear} on $X \times X^*$ into a vector space $P_h$ commutes with the process of semi-discretization of the evolution equation \eqref{Semilinear Evolution Equation} into a vector space $X_h$ followed by adjoining if and only if it is equivalent to the adjoint system on $X_h \times X_h^*$, equipped with a duality pairing $\langle\cdot,\cdot\rangle_h: X_h^* \times X_h \rightarrow \mathbb{R}$, formed from the semi-discrete ODE \eqref{eq:method-of-lines-semi-discretization}; namely,
    \begin{subequations}\label{eq:method-of-lines-adjoint}
    \begin{align}
        \frac{d}{dt} \mathbf{q} &= K_h[A,f](t,\mathbf{q}), \label{eq:method-of-lines-adjoint-a} \\
        \frac{d}{dt} \mathbf{z} &= - [D_{\mathbf{q}}K_h[A,f](t,\mathbf{q})]^{*h} \mathbf{z}. \label{eq:method-of-lines-adjoint-b}
    \end{align}
    \end{subequations}
    Furthermore, given a method-of-lines semi-discretization \eqref{eq:method-of-lines-semi-discretization} of the evolution equation, the dual semi-discretization \eqref{eq:method-of-lines-adjoint} is the unique method-of-lines semi-discretization into $X_h \times X_h^*$, equipped with the duality pairing $\langle\cdot,\cdot\rangle_h$, of the adjoint system covering \eqref{eq:method-of-lines-semi-discretization} on the interval $[0,t_f]$ such that the adjoint-variational conservation law holds,
    $$ \frac{d}{dt} \langle \mathbf{z}(t), \delta \mathbf{q}(t) \rangle_h = 0, $$
    where $\delta \mathbf{q}(t)$ is the solution of the variational equation associated with \eqref{eq:method-of-lines-semi-discretization},
    $$ \frac{d}{dt} \delta \mathbf{q} = [D_{\mathbf{q}}K_h[A,f](t,\mathbf{q})] \delta \mathbf{q},$$
    with arbitrary but fixed initial condition $\delta \mathbf{q}(t_0)= \delta \mathbf{q}_0$ for any $t_0 \in [0,t_f)$ and terminal condition $\mathbf{z}(t_f)=\mathbf{z}_f$.
    \begin{proof}
        The first statement of the theorem simply follows from the definitions and a direct calculation that \eqref{eq:method-of-lines-adjoint-b} is the adjoint equation associated with \eqref{eq:method-of-lines-adjoint-a}.

        For the second statement, clearly \eqref{eq:method-of-lines-adjoint} satisfies the above adjoint-variational conservation law by Proposition \ref{prop:ManifoldQuadraticInvariant}.
        
        To show that it is unique, suppose we have another method-of-lines semi-discretization of the adjoint system covering the semi-discretization of the evolution equation, i.e., we have a semi-discretization of the continuous adjoint system of the form
        \begin{align*}
        \frac{d}{dt} \mathbf{q} &= K_h[A,f](t,\mathbf{q}),\\
        \frac{d}{dt} \tilde{\mathbf{z}} &= L_h(t,\mathbf{q},\tilde{\mathbf{z}}),
        \end{align*}
        satisfying
        $$ \frac{d}{dt} \langle \tilde{\mathbf{z}}(t), \delta \mathbf{q}(t) \rangle_h = 0. $$ Since both semi-discretizations satisfy the Type II boundary conditions fixing $\delta \mathbf{q}(t_0) = \delta \mathbf{q}_0$ and $\mathbf{z}(t_f) = \mathbf{z}_f = \tilde{\mathbf{z}}(t_f)$, we have by integrating their respective quadratic conservation laws from $t_0$ to $t_f$,
        \begin{align*}
            \langle \tilde{\mathbf{z}}(t_0), \delta \mathbf{q}_0\rangle_h = \langle \tilde{\mathbf{z}}(t_f), \delta \mathbf{q}(t_f) \rangle_h = \langle \mathbf{z}_f, \delta \mathbf{q}(t_f) \rangle_h = \langle \mathbf{z}(t_f), \delta \mathbf{q}(t_f) \rangle_h = \langle \mathbf{z}(t_0), \delta \mathbf{q}_0 \rangle_h.
        \end{align*}
        In particular, since $\delta \mathbf{q}_0$ is arbitrary, we have
        $$ \tilde{\mathbf{z}}(t_0) = \mathbf{z}(t_0). $$
        Since $t_0 \in [0,t_f)$ is arbitrary, we have $\mathbf{z}(t) = \tilde{\mathbf{z}}$ for all $t \in [0,t_f]$. Thus,
        $$ L_h(t,\mathbf{q}(t),\mathbf{z}(t)) = L_h(\mathbf{q}(t),\tilde{\mathbf{z}}(t)) = \frac{d}{dt} \tilde{\mathbf{z}}(t) = \frac{d}{dt} \mathbf{z}(t) = - [D_{\mathbf{q}}K_h[A,f](t,\mathbf{q}(t))]^{*h} \mathbf{z}(t), $$
        i.e., they are the same semi-discretization. 
    \end{proof}
\end{theorem}

\begin{remark}
    Note that in the above theorem, we allow the semi-discretization of the adjoint system to map into a vector space $P_h$, not necessarily $X_h \times X_h^*$, and it suffices to require that the semi-discretization of the adjoint system on $P_h$ is equivalent to the adjoint system on $X_h \times X_h^*$ formed from \eqref{eq:method-of-lines-semi-discretization}, i.e., there exists an invertible transformation $\Phi_{P_h}: P_h \rightarrow X_h \times X_h^*$ mapping the corresponding semi-discrete adjoint systems to each other. More precisely, $P_h$ must have a Hamiltonian structure given by pulling back the Hamiltonian structure on $X_h \times X_h^*$,
    \begin{align*}
        \Omega_{P_h} &= \Phi_{P_h}^* \Omega_h, \\
        H_{P_h} &= \Phi_{P_h}^* H_h,
    \end{align*}
    where $\Omega_h$ is the canonical symplectic form on $X_h \times X_h^*$ and $H_h = \langle \mathbf{z}, K_h[A,f](\mathbf{q})\rangle_h$ is the adjoint Hamiltonian associated with the system \eqref{eq:method-of-lines-adjoint} (here, we assume that $f$ is time-independent for simplicity, but an analogous statement holds for the time-dependent case by considering the associated Cartan forms, as described in the previous Galerkin semi-discretization case).

    In the literature, it is often the case that the semi-discretization of the adjoint system is not formulated from the cotangent bundle $X \times X^*$ to $X_h \times X_h^*$. When $X$ is a Hilbert space, often the identification of $X^*$ with $X$ via the Riesz representation theorem is utilized to write the adjoint system as a system on $X \times X$. For example, this is done in \cite{NoWa2007} for the adjoint system associated with Burgers' equation, where the same (piecewise linear finite element) method-of-lines semi-discretization is used for both the state dynamics in the variable $y$ and the adjoint equation in the variable $p$, resulting in a semi-discrete ODE on $X_h \times X_h$. Thus, in terms of the notation introduced in the above theorem, the semi-discretization in \cite{NoWa2007} utilizes $P_h = X_h \times X_h$. With this example in mind, note that the uniqueness result from the above theorem applies more generally to dual systems \cite{TVS1999} of the form $P_h = X_h \times Y_h$ equipped with a non-degenerate bilinear form $b: X_h \times Y_h \rightarrow \mathbb{R}$; the prototypical examples are $P_h = X_h \times X_h^*$ equipped with a duality pairing and $P_h = X_h \times X_h$ equipped with an inner product. Note that the non-degenerate bilinear form $b$ defines an isomorphism $b^\flat : Y_h \rightarrow X_h^*$ and hence, induces a symplectic structure on the dual system.

    As an immediately corollary to Theorem \ref{thm:method-of-lines-adjoint-comm}, we have that any method-of-lines semi-discretization of an infinite-dimensional adjoint system which corresponds to the adjoint of a method-of-lines semi-discretization of the evolution equation must necessarily have a Hamiltonian structure on $P_h$, since it must be equivalent to the Hamiltonian structure on $X_h \times X_h^*$. Heuristically, for semi-discretize-then-optimize and optimize-then-semi-discretize methods to commute, the optimize-then-semi-discretize method must necessarily preserve the Hamiltonian structure. 

    Note that the above uniqueness depends on the choice of the duality pairing; for example, as we have seen explicitly in the Galerkin semi-discretization case, there are two semi-discretizations on $X_h \times X_h^*$ (which we identified with $\mathbb{R}^N \times \mathbb{R}^{N*}$ using a basis) which satisfy the adjoint quadratic conservation law; namely, one with respect to the standard duality pairing on $\mathbb{R}^N$ and one with respect to the duality pairing $\langle\cdot,\cdot\rangle_M$. Although the two systems arising from different duality pairings are equivalent via a similarity transformation, this subtle distinction becomes important when moving to the fully discrete setting by incorporating time integration, as we will explain in Section \ref{section:naturality}.
\end{remark}

\subsection{Time Integration}\label{section:time-integration}
To completely discretize a semi-discrete system, we have to further integrate the system in time and thus consider in this section the relation between time integration and forming adjoints. Note that the results of this section hold for general adjoint systems for ODEs, not just those that arise from semi-discretization; as such, we will consider an ODE on a finite-dimensional manifold $V$ for greater generality. To begin, we will recall some facts about maps and time integration for ODEs. 

Consider an ODE $\dot{\mathbf{y}} = \mathbf{g}(t,\mathbf{y})$ on a finite-dimensional manifold $V$. We will denote the fiber-wise duality pairing on $T_\mathbf{y}V^* \times T_\mathbf{y}V$ as $\langle\cdot,\cdot\rangle$ and the adjoint of an operator $B$ as $B^*$.

Let $\Phi:V \rightarrow V$ be a (local) diffeomorphism. Recall that the \textit{tangent lift} of $\Phi$, denoted $T\Phi: TV \rightarrow TV$, is defined by
$$ T\Phi(\mathbf{v}_{\mathbf{y}}) = T_{\mathbf{y}}\Phi(\mathbf{v}_{\mathbf{y}}) \in T_{\Phi(\mathbf{y})}V \text{ for } \mathbf{v}_{\mathbf{y}} \in T_\mathbf{y}V, $$
where $T_\mathbf{y}\Phi(\mathbf{v}_\mathbf{y})$ is the linearization of $\Phi$ at $\mathbf{y}$, which is represented as the Jacobian of $\Phi$ in a local chart. This induces a dual map on the cotangent spaces, for $\mathbf{a} \in T^*_{\Phi(\mathbf{y})}V$, by
\begin{equation}\label{eq:dual-tangent}\langle T^*_\mathbf{y}\Phi(\mathbf{a}), \mathbf{v}_\mathbf{y} \rangle = \langle \mathbf{a}, T_\mathbf{y}\Phi(\mathbf{v}_\mathbf{y}) \rangle \text{ for all } \mathbf{v}_\mathbf{y} \in T_\mathbf{y}V. 
\end{equation}
We then define the cotangent lift of $\Phi$ to be $T^*\Phi^{-1}$, which is a (local) vector bundle morphism $T^*\Phi^{-1}: T^*V \rightarrow T^*V$. Furthermore, it is a (local) symplectomorphism, with $T^*V$ equipped with its canonical symplectic form. 

Now, consider a one-step method for the ODE $\dot{\mathbf{y}} = \mathbf{g}(t,\mathbf{y})$, which is specified by a map 
$$ \Phi_{\Delta t}[n,\mathbf{g}]: V \rightarrow V $$
which depends on the current time $t_n = t_0 + n\Delta t$ and the vector field $\mathbf{g}$ defining the ODE; for the discussion that follows, we think of $n$ as fixed and thus, $\Phi_{\Delta t}[n,\mathbf{g}]$ defines a map from $V$ to $V$. The one-step method is given by
$$ \mathbf{y}_{n+1} = \Phi_{\Delta t}[n,\mathbf{g}] (\mathbf{y}_n). $$
We will assume that $\Phi_{\Delta t}[n,\mathbf{g}]$ is a local diffeomorphism, which is generally true by an implicit function argument, given differentiability of $\mathbf{g}$ in its second argument and a sufficiently small time step. Thus, we can define its cotangent lift $T^*\Phi^{-1}_{\Delta t}[n,\mathbf{g}]$. As the cotangent lift is a symplectomorphism, we have that the method
\begin{align*}
    \mathbf{y}_{n+1} &= \Phi_{\Delta t}[n,\mathbf{g}] (\mathbf{y}_n), \\
    \mathbf{p}_{n+1} &= T^*\Phi^{-1}_{\Delta t}[n,\mathbf{g}] \mathbf{p}_n,
\end{align*}
is a symplectic method on $T^*V$, i.e., $\langle d\mathbf{y}_{n+1} \wedge d\mathbf{p}_{n+1} \rangle = \langle d\mathbf{y}_{n} \wedge d\mathbf{p}_{n} \rangle$. With the solution curve $\{y_n\}$ for the state variable fixed, we interpret the cotangent lift as a one-step map $T^*\Phi^{-1}_{\Delta t}[n,\mathbf{g}]: T_{\mathbf{y}_{n}}^*V \rightarrow T_{\mathbf{y}_{n+1}}^*V$. Note that this method can be thought of as a one-step approximation of the adjoint system, since the continuous-time flow of the adjoint system is given by the cotangent lift of the flow of the state dynamics (as discussed in Section \ref{section:geometry-adjoints-review}). Also note that for backpropagation where the terminal value for the adjoint variable is specified, the second equation above is more naturally expressed as a map $\mathbf{p}_{n+1} \mapsto \mathbf{p}_n$ given by $\mathbf{p}_{n} = T^*\Phi_{\Delta t}[n,\mathbf{g}] \mathbf{p}_{n+1}.$


\begin{remark}
    We will refer to the \textit{cotangent lifted one-step method} as the combined integrator in both the $y$ and $p$ variables, as well as just the integrator in the $p$ variable (which depends on the $y$ variable); it will be clear in context which is meant. 
\end{remark}

It is well-known that integration of the state dynamics by a Runge--Kutta (RK) method and then adjoining is equivalent to first adjoining and then integrating by the associated symplectic partitioned RK method (see \cite{Sa2016, BoVa2006, TrLe2024}). We extend this result to all one-step time integration methods. 

\begin{theorem}\label{thm:time-int-adjoint-comm}
Time integration by a one-step method and adjoining commute, where time integration of the state ODE is given by a one-step method $\Phi_{\Delta t}[n,\mathbf{g}]$, time integration of the adjoint system is given by the cotangent lift of the one-step method, adjoining the ODE is given by the usual adjoint ODE system \cite{TrLe2024}, and the adjoint of the discrete system is as defined in \cite{TrLe2024}. That is, the following diagram commutes.
\begin{equation}
\begin{tikzcd}[column sep=20ex, row sep =10ex]\label{diagram:time-int-adjoint-comm}
	\dot{\mathbf{y}} = \mathbf{g}(t,\mathbf{y}) \arrow[r, "\text{Adjoint}"] \arrow[d, "\Phi_{\Delta t}{[n,\mathbf{g}]}"] & \begin{tabular}{c} $\dot{\mathbf{y}} = \mathbf{g}(t,\mathbf{y})$ \\ $\dot{\mathbf{p}} = -[D_\mathbf{y}\mathbf{g}(t,\mathbf{y})]^* \mathbf{p}$ \end{tabular} \arrow[d, "(\Phi_{\Delta t}{[n,\mathbf{g}]}{,}T^*\Phi^{-1}_{\Delta t}{[n,\mathbf{\mathbf{g}}]})"] \\ 
    \mathbf{y}_{n+1} = \Phi_{\Delta t}{[n,\mathbf{g}]}(\mathbf{y}_n) \arrow[r, "\text{Adjoint}"] & \begin{tabular}{c} $\mathbf{y}_{n+1} = \Phi_{\Delta t}{[n,\mathbf{g}]}(\mathbf{y}_n) $ \\ $\mathbf{p}_{n+1} = T^*\Phi^{-1}_{\Delta t}{[n,\mathbf{g}]}(\mathbf{p}_{n}) $ \end{tabular}
\end{tikzcd}
\end{equation}
    Furthermore, the cotangent lifted method is the unique one-step method $\mathbf{p}_{n} \mapsto \mathbf{p}_{n+1}$ satisfying
    \begin{equation}\label{eq:cotangent-lift-characterization}
        \langle \mathbf{p}_{n+1},\delta \mathbf{y}_{n+1} \rangle = \langle \mathbf{p}_{n}, \delta \mathbf{y}_n\rangle, 
    \end{equation}
    where $\delta \mathbf{y}_{n+1}$ solves the variational equation associated with the one-step method
    $$ \delta \mathbf{y}_{n+1} = T_{\mathbf{y}_n}\Phi_{\Delta t}[n,\mathbf{g}] \delta \mathbf{y}_n, $$
    for arbitrary $\delta \mathbf{y}_n$ and arbitrary right-hand-side of the ODE $g$. Furthermore, an alternative representation of the cotangent lifted one-step method is 
    \begin{equation}\label{eq:alternate-cotangent-lift-representation}
        \mathbf{p}_{n} = \left(\frac{\delta \mathbf{y}_{n+1}}{\delta \mathbf{y}_n}\right)^*\mathbf{p}_{n+1}.
    \end{equation}
\begin{proof}
    To prove the first statement, we have to show that the discrete adjoint system $\mathbf{y}_{n+1} = \Phi_{\Delta t}{[n,\mathbf{g}]}(\mathbf{y}_n)$ yield the method given by the cotangent lift of $\Phi_{\Delta t}{[n,\mathbf{g}]}$ applied to the continuous adjoint system. For a state variable $\mathbf{y}_n$ and an adjoint variable $\mathbf{p}_{n+1}$, we define the discrete action by 
    $$ \mathbb{S}(\mathbf{y}_n,\mathbf{p}_{n+1}) = \langle \mathbf{p}_{n+1},\Phi_{\Delta t}{[n,\mathbf{g}]}(\mathbf{y}_n)\rangle. $$
    The discrete adjoint system is then given by \cite{TrLe2024}
    \begin{align*}
        \mathbf{y}_{n+1} &= \frac{\delta}{\delta \mathbf{p}_{n+1}}\mathbb{S}(\mathbf{y}_n,\mathbf{p}_{n+1}), \\
        \mathbf{p}_{n} &= \frac{\delta}{\delta \mathbf{y}_n} \mathbb{S}(\mathbf{y}_n,\mathbf{p}_{n+1}).
    \end{align*}
    The first equation is simply the one-step method for the state ODE,
    $$\mathbf{y}_{n+1} = \frac{\delta}{\delta \mathbf{p}_{n+1}}\mathbb{S}(\mathbf{y}_n,\mathbf{p}_{n+1}) = \Phi_{\Delta t}{[n,\mathbf{g}]}(\mathbf{y}_n).$$
    The second equation can be computed, for an arbitrary variation $\delta\mathbf{y} \in T_{\mathbf{y}_n}V$ specified by a one-parameter family of curves $\mathbf{y}_n^\epsilon$ on $V$ such that $\mathbf{y}_n^0 = \mathbf{y}_n$ and $\frac{d}{d\epsilon}\Big|_{\epsilon=0}\mathbf{y}_n^\epsilon = \delta \mathbf{y}$,
    \begin{align*}
        \langle \mathbf{p}_{n}, \delta \mathbf{y}\rangle &= \frac{d}{d\epsilon}\Big|_{\epsilon = 0} \mathbb{S}(\mathbf{y}_n^\epsilon,\mathbf{p}_{n+1}) = \left\langle \mathbf{p}_{n+1}, \frac{d}{d\epsilon}\Big|_{\epsilon = 0}\Phi_{\Delta t}{[n,\mathbf{g}]}(\mathbf{y}_n^\epsilon)\right\rangle \\
        &= \langle \mathbf{p}_{n+1},T_{\mathbf{y}_n}\Phi_{\Delta t}{[n,\mathbf{g}]} \delta \mathbf{y}\rangle = \langle T^*_{\mathbf{y}_n}\Phi_{\Delta t}{[n,\mathbf{g}]} \mathbf{p}_{n+1}, \delta \mathbf{y}\rangle.
    \end{align*}
    Since this holds for any $\delta \mathbf{y} \in T_{\mathbf{y}_n}V$, the above equation is equivalent to
    $$ \mathbf{p}_{n} = T^*\Phi_{\Delta t}{[n,\mathbf{g}]} \mathbf{p}_{n+1}, $$
    or, equivalently,
    $$ \mathbf{p}_{n+1} = T^*\Phi^{-1}_{\Delta t}{[n,g]}\mathbf{p}_{n}, $$
    as was to be shown.

    For the second statement of the theorem that the cotangent lift of the one-step method satisfies equation \eqref{eq:cotangent-lift-characterization}, we use that the cotangent lifted method is given by $\mathbf{p}_{n} = T^*_{\mathbf{y}_n}\Phi_{\Delta t}{[n,\mathbf{g}]} \mathbf{p}_{n+1}$, from which,
        \begin{align*}
            \langle \mathbf{p}_{n}, \delta \mathbf{y}_n \rangle = \langle T^*_{\mathbf{y}_n}\Phi_{\Delta t}{[n,\mathbf{g}]} \mathbf{p}_{n+1}, \delta \mathbf{y}_n \rangle = \langle \mathbf{p}_{n+1}, T_{\mathbf{y}_n}\Phi_{\Delta t}{[n,\mathbf{g}]} \delta \mathbf{y}_n \rangle = \langle \mathbf{p}_{n+1}, \delta \mathbf{y}_{n+1}\rangle .
        \end{align*}
        
        For uniqueness, suppose there are two one-step methods $\mathbf{p}_{n} \mapsto \mathbf{p}_{n+1}$ and $\mathbf{p}_{n} \mapsto \tilde{\mathbf{p}}_{n+1}$ satisfying \eqref{eq:cotangent-lift-characterization}, i.e.,
        $$ \langle \mathbf{p}_{n+1}, \delta \mathbf{y}_{n+1}\rangle = \langle \mathbf{p}_{n}, \delta \mathbf{y}_n \rangle = \langle \tilde{\mathbf{p}}_{n+1}, \delta \mathbf{y}_{n+1}\rangle.  $$
        Using $\delta \mathbf{y}_{n+1} = T_{\mathbf{y}_n}\Phi_{\Delta t}{[n,\mathbf{g}]} \delta \mathbf{y}_n$, we have
        $$ \langle \mathbf{p}_{n+1},T_{\mathbf{y}_n}\Phi_{\Delta t}{[n,\mathbf{g}]} \delta \mathbf{y}_n\rangle = \langle \tilde{\mathbf{p}}_{n+1}, T_{\mathbf{y}_n}\Phi_{\Delta t}{[n,\mathbf{g}]} \delta \mathbf{y}_n\rangle. $$
        Equivalently,
        $$ \langle (T^*_{\mathbf{y}_n}\Phi_{\Delta t}{[n,\mathbf{g}]})(\mathbf{p}_{n+1} - \tilde{\mathbf{p}}_{n+1}), \delta \mathbf{y}_n \rangle = 0. $$
        Since $\delta \mathbf{y}_n$ is arbitrary, we have $(T^*_{\mathbf{y}_n}\Phi_{\Delta t}{[n,\mathbf{g}]})(\mathbf{p}_{n+1} - \tilde{\mathbf{p}}_{n+1}) = 0$. Finally, since the one-step method is a local diffeomorphism, the operator $T^*_{\mathbf{y}_n}\Phi_{\Delta t}{[n,\mathbf{g}]}$ has trivial kernel, and hence, $\mathbf{p}_{n+1} = \tilde{\mathbf{p}}_{n+1}$.

        Finally, the representation \eqref{eq:alternate-cotangent-lift-representation} follows from observing that since $\mathbf{y}_{n+1}$ is a function of $\mathbf{y}_n$, their variations are related by 
        \begin{equation}
            \delta \mathbf{y}_{n+1} = \frac{\delta \mathbf{y}_{n+1}}{\delta \mathbf{y}_n} \delta \mathbf{y}_n. \label{eq:tangent-transfer-rep}
        \end{equation}
        Substituting this into \eqref{eq:cotangent-lift-characterization} yields
        $$ \langle \mathbf{p}_{n}, \delta \mathbf{y}_n\rangle =   \langle \mathbf{p}_{n+1},\delta \mathbf{y}_{n+1} \rangle =  \left\langle \mathbf{p}_{n+1},\frac{\delta \mathbf{y}_{n+1}}{\delta \mathbf{y}_n} \delta \mathbf{y}_n\right\rangle = \left\langle \left(\frac{\delta \mathbf{y}_{n+1}}{\delta \mathbf{y}_n}\right)^*\mathbf{p}_{n+1}, \delta \mathbf{y}_n\right\rangle. $$
        Again, $\delta \mathbf{y}_n$ is arbitrary and thus, equation \eqref{eq:alternate-cotangent-lift-representation} holds.
\end{proof}
\end{theorem}

Analogous to the discussion of semi-discretization in Section \ref{section:semi-discretization}, this commutative diagram leads to a non-trivial structural condition that an OtD method (where here discretize refers to one-step time integration) must satisfy in order to be equivalent to a DtO method; namely, the one-step method used in the OtD method must be a symplectic integrator and, in particular, equivalent to the cotangent lift of a one-step method (here, by equivalent, we mean that we are allowing for equivalent representations of the same time integration scheme, such as the concept of reducibility in the context of Runge--Kutta methods \cite{Hairer.2006}).

The previous theorem shows that the cotangent lifted method is the unique one-step method covering the one-step method for the state dynamics and satisfying the adjoint quadratic conservation law \eqref{eq:cotangent-lift-characterization}. This explains the observed discrepancy in the literature for discrete gradients produced by DtO and OtD methods for ODEs, as discussed in Section \ref{section:DtO-vs-OtD}, since utilizing a time integration scheme in an OtD method which is not the cotangent lifted method cannot satisfy the discrete adjoint quadratic conservation law and hence, cannot produce exact discrete gradients.

\textbf{Order of the Cotangent Lifted Method.} In the context of adjoint sensitivity analysis, since the conservation law \eqref{eq:cotangent-lift-characterization} holds, the cotangent lifted method produces the exact discrete gradient \cite{Sa2016} for the discrete minimization problem
$$ \min_{\mathbf{y}^*} C\, (\mathbf{y}_N) \text{ such that } \mathbf{y}_{n+1} = \Phi_{\Delta t}{[n,\mathbf{g}]}(\mathbf{y}_n), n =0,\dots,N-1,\ \mathbf{y}_0 = \mathbf{y}^*, $$
where $N$ is the index corresponding to $t_f$. We now address the question of adjoint consistency, i.e., how well the discrete gradient produced from the cotangent lifted method approximates the continuous-time gradient, corresponding to the continuous-time minimization problem
\begin{align*}
    \min_{\mathbf{y}^*} C\, (\mathbf{y}(t_f)) \text{ such that } \dot{\mathbf{y}}(t) = \mathbf{g}(t,\mathbf{y}), \mathbf{y}(0) = \mathbf{y}^*.
\end{align*}

For this section, we will assume that $V$ is a finite-dimensional vector space with norm $\|\cdot\|_V$. In essence, we would like to transfer the order of accuracy of the one-step method to the order of accuracy of its cotangent lift. To do this, we will need an additional assumption on the one-step method; namely, that taking variations and applying the one-step method commute. To be more precise, we introduce the following definitions.

The \textit{variational system} associated with the ODE $\dot{\mathbf{y}} = \mathbf{g}(t,\mathbf{y})$ on $V$ is given by the ODE together with its variational equation, i.e.,
\begin{align} \label{eq:variational-system-cont}
    \frac{d}{dt} \begin{pmatrix} \mathbf{y} \\ \delta \mathbf{y} \end{pmatrix} &= \begin{pmatrix} \mathbf{g}(t,\mathbf{y}) \\ D_\mathbf{y}\mathbf{g}(t,\mathbf{y})\delta \mathbf{y} \end{pmatrix} =: \widetilde{\mathbf{g}} \left(t, \begin{pmatrix} \mathbf{y} \\ \delta \mathbf{y} \end{pmatrix} \right).
\end{align}
We denote the right hand side of the variational system as $\widetilde{g}$, viewed as an ODE on $V \times V$. 

Let $\Phi_{\Delta t}[\cdot,\cdot]$ be a one-step method. The \textit{variational system} associated with the one-step method is given by the one-step method together with its variational equation, i.e., 
\begin{subequations}\label{eq:variational-system-one-step}
\begin{align}
    \mathbf{y}_{n+1} &= \Phi_{\Delta t}{[n,\mathbf{g}]}(\mathbf{y}_n), \label{eq:variational-system-one-step-a}\\
    \delta \mathbf{y}_{n+1} &= T_{\mathbf{y}_n} \Phi_{\Delta t}{[n,\mathbf{g}]} \delta \mathbf{y}_n. \label{eq:variational-system-one-step-b}
\end{align}
\end{subequations}
We say that the one-step method is \textit{variationally equivariant} if the one-step method applied to the variational system associated with the ODE \eqref{eq:variational-system-cont} is the same as the variational system associated with the one-step method \eqref{eq:variational-system-one-step}. That is,
$$ \Phi_{\Delta t}[n, \widetilde{\mathbf{g}}] \begin{pmatrix} \mathbf{y}_n \\ \delta \mathbf{y}_n \end{pmatrix} = \begin{pmatrix} \Phi_{\Delta t}{[n,\mathbf{g}]}(\mathbf{y}_n) \\ T_{\mathbf{y}_n} \Phi_{\Delta t}{[n,\mathbf{g}]} \delta \mathbf{y}_n \end{pmatrix}. $$
Informally, a one-step method is variationally equivariant if applying the one-step method and taking variations commute. For example, this is true for Runge--Kutta methods (RK) \cite{Hairer.2006}, more generally for Generalized Additive Runge--Kutta methods (GARK) \cite{NaRoSa2020}, and exponential Runge--Kutta methods.

We can now relate the order of accuracy for the discrete gradients obtained from the cotangent lifted method to the order of accuracy of the method for the state dynamics. The following analysis closely mirrors the corresponding result for RK and GARK schemes \cite{Sa2006, NaRoSa2020}.

Define the \textit{solution sensitivity matrix} associated with an ODE $\dot{\mathbf{y}} = \mathbf{g}(t,\mathbf{y})$ on $V$ as
$$ S_{t_2,t_1}(\mathbf{y}(t_1)) = \frac{\delta \mathbf{y}(t_2)}{\delta \mathbf{y}(t_1)},\ t_2 \geq t_1,$$
where $\mathbf{y}(t)$ is the exact solution of the ODE. Note we view the solution sensitivity matrix as a linear mapping $S_{t_2,t_1}(\mathbf{y}(t_1)) : V \rightarrow V$. 

\begin{prop}\label{prop:cotangent-order}
    Let $\Phi_{\Delta t}$ be a variationally equivariant one-step method. Assume that the ODE $\dot{\mathbf{y}} = \mathbf{g}(t,\mathbf{y})$ is smooth and has a smooth solution. Furthermore, assume sufficient regularity such that the solution sensitivity matrix and the derivative of the cost function are Lipschitz,
    \begin{align}
        \|S_{t_2,t_1}(\mathbf{y}) - S_{t_2,t_1}(\mathbf{y}')\|_{op} &\leq c \|\mathbf{y}-\mathbf{y}'\|_V, \label{eq:solution-estimate} \\
        \|D\text{C}\,(\mathbf{y}) - D\text{C}\,(\mathbf{y}')\|_{V^*} &\leq c \|\mathbf{y}-\mathbf{y}'\|_V, \label{eq:gradient-estimate}
    \end{align}
     for all $\mathbf{y},\mathbf{y}' \in V$ and $t_2 \geq t_1$ (where $\|\cdot\|_V$ denotes a norm on $V$, $\|\cdot\|_{V^*}$ the associated dual norm, and $\|\cdot\|_{op}$ the operator norm; this result is independent of the norm since $V$ is finite-dimensional.) Assume that the one-step method $\mathbf{y}_{n+1} = \Phi_{\Delta t}{[n,\mathbf{g}]}(\mathbf{y}_n)$ converges with order $r$,
    $$ \mathbf{y}_n - \mathbf{y}(t_n) = \mathcal{O}(\Delta t^r).$$
    Then, the cotangent lifted method $\mathbf{p}_{n} = T^*_{\mathbf{y}_n}\Phi_{\Delta t}{[n,\mathbf{g}]}\mathbf{p}_{n+1}$ with terminal condition $\mathbf{p}_{n} = DC(\mathbf{y}_n)$ approximates the solution of the continuous-time adjoint equation with order $r$
    $$ \mathbf{p}_{n} - \mathbf{p}(t_n) = \mathcal{O}(\Delta t^r), $$
    where $\mathbf{p}(t)$ is the exact solution of
    $$ \dot{\mathbf{p}} = -[D\mathbf{g}(t,\mathbf{y})]\mathbf{p},\ \mathbf{p}(t_f) = DC\,(\mathbf{y}(t_f)). $$
    \begin{proof}
        The solution $\delta \mathbf{y}(t)$ for the second component of the continuous-time variational system \eqref{eq:variational-system-cont} is simply propagated by the solution sensitivity matrix, $\delta \mathbf{y}(t_2) = S_{t_2,t_1}(\mathbf{y}(t_1)) \delta \mathbf{y}(t_1)$ \cite{NaRoSa2020}.

        At the discrete level, the solution of the variational equation associated with the one-step method $\delta \mathbf{y}_{n+1} = T_{\mathbf{y}_n} \Phi_{\Delta t}{[n,\mathbf{g}]} \delta \mathbf{y}_n$, from equation \eqref{eq:tangent-transfer-rep}, can be represented $\delta \mathbf{y}_{n_2} = S^{\Delta t}_{n_2,n_1}(\mathbf{y}_{n_1}) \delta \mathbf{y}_{n_1}$, where $S^{\Delta t}_{n_2,n_1}(\mathbf{y}_{n_1})$ is the numerical solution sensitivity matrix
        $$ S^{\Delta t}_{n_2,n_1}(\mathbf{y}_{n_1}) = \frac{\delta \mathbf{y}_{n_2}}{\delta \mathbf{y}_{n_1}}.$$
        By assumption, the one-step method is variationally equivariant, so the solution $\{\mathbf{y}_n, \delta \mathbf{y}_n\}$ of the variational system associated with the one-step method is equivalently given by the one-step method applied to the continuous-time variational system. In particular, this solution then inherits the order of accuracy of the one-step method. Furthermore, the stability for $\{\delta \mathbf{y}_n\}$ is precisely the linear stability of the one-step method. Thus, we have convergence of order $r$ for $\{\delta \mathbf{y}_n\}$,
        $$ \delta \mathbf{y}_n = \delta \mathbf{y}(t_n) + \mathcal{O}(\Delta t^r). $$
        In particular, this implies that the numerical sensitivity matrix approximates the continuous sensitivity matrix to order $r$, since for arbitrary $\mathbf{v} \in V$ and initial condition $\delta \mathbf{y}_{n_1} = \mathbf{v} = \delta \mathbf{y}(t_1)$, we have
        $$ \mathcal{O}(\Delta t^r) = \delta \mathbf{y}(t_2) - \delta \mathbf{y}_{n_2} = \Big( S_{t_2,t_1}(\mathbf{y}(t_1)) - S^{\Delta t}_{n_2,n_1} (\mathbf{y}(t_1))\Big) \mathbf{v}.$$
        Thus,
        $$ S_{t_2,t_1}(\mathbf{y}) - S^{\Delta t}_{n_2,n_1} (\mathbf{y}) = \mathcal{O}(\Delta t^r).$$ From Theorem \eqref{thm:time-int-adjoint-comm}, we know that 
        $$ \mathbf{p}_{n} = \left( \frac{\delta \mathbf{y}_{n+1}}{\delta \mathbf{y}_n}\right)^*\mathbf{p}_{n+1},$$ and hence
        $$ \mathbf{p}_{n} = \left(S^{\Delta t}_{N,n}\right)^* \mathbf{p}_{n} = \left(S^{\Delta t}_{N,n}(\mathbf{y}_n) \right)^* DC (\mathbf{y}_n), $$
        whereas the continuous-time adjoint variable similarly satisfies
        $$ \mathbf{p}(t_n) = \frac{\delta C(\mathbf{y}(t_f))}{\delta \mathbf{y}(t_n)} = \left(S^{\Delta t}_{t_f,t_n}(\mathbf{y}(t_n))\right)^*DC(\mathbf{y}(t_f)).$$
        Combining the above, we obtain
        \begin{align*}
            \mathbf{p}_{n} - \mathbf{p}(t_n) &= \Big(S^{\Delta t}_{N,n}(\mathbf{y}_n) \Big)^* DC (\mathbf{y}_n) - \Big(S_{t_f,t_n}(\mathbf{y}(t_n))\Big)^*DC(\mathbf{y}(t_f)) \\
            &= \Big(S^{\Delta t}_{N,n}(\mathbf{y}_n) \Big)^* (DC (\mathbf{y}_n) - DC(\mathbf{y}(t_f))) \\
            & \qquad + \Big(S^{\Delta t}_{N,n}(\mathbf{y}_n) - S_{t_f,t_n} (\mathbf{y}_n)\Big)^* DC (\mathbf{y}(t_f)) + \Big(S_{t_f,t_n}(\mathbf{y}_n) - S_{t_f,t_n}(\mathbf{y}(t_n)) \Big)^* DC(\mathbf{y}(t_f)).
        \end{align*}
        The result now immediately follows.
    \end{proof}
\end{prop}

To prove the above proposition, we used variational equivariance to transfer the order of the one-step method to the order of the linearization of the one-step method, which is equivalently the one-step method of the linearization. Note that variational equivariance is a sufficient condition to be able to transfer the order, i.e., variational equivariance is a sufficient condition for adjoint consistency of the cotangent lifted method. However, even without a variationally equivariant method, if one knows that the linearization of the one-step method retains the order of the one-step method, then the preceding result still follows. The contrapositive implies that if we do not have adjoint consistency, then the one-step method is not variationally equivariant, as the following example shows.

\begin{example}
As a counterexample of where the above proposition fails without variational equivariance, in \cite{AlSa2009}, it is shown that discrete adjoints for adaptive time-stepping methods do not retain the consistency order of the method for the state dynamics, leading in the worst case to $\mathcal{O}(1)$ errors in the gradient, i.e., the discrete adjoint of an adaptive time-stepping method is not adjoint consistent. From our perspective, we can understand this from the fact that an adaptive time-stepping method is not variationally equivariant, even if the underlying one-step method is. 

To see this explicitly, consider the time-independent ODE $\dot{\mathbf{y}}(t) = \mathbf{g}(t,\mathbf{y}(t)), \mathbf{y}(t_0)=\mathbf{y}_0$ expressed in time-independent form as
\begin{align}\label{eq:time-ind-form-ode}
    \frac{d}{dt} \begin{pmatrix} \mathbf{y}(t) \\ s(t) \end{pmatrix} = \begin{pmatrix} \mathbf{g}(s(t), \mathbf{y}(t)) \\ 1 \end{pmatrix},\ \begin{pmatrix} \mathbf{y}(t_0) \\ s(t_0) \end{pmatrix} = \begin{pmatrix} \mathbf{y}_0 \\ t_0 \end{pmatrix}. 
\end{align}
For simplicity in presentation, we consider an adaptive forward Euler method applied to the above system, although a similar discussion follows similarly for other one-step methods. Note that the underlying forward Euler method is variationally equivariant, but, as we will see, the adaptive method is not. Applied to the above system, the adaptive forward Euler method is
\begin{subequations}\label{eq:adaptive-euler}
\begin{align}\label{eq:adaptive-euler-a}
    \begin{pmatrix} \mathbf{y}_{n+1} \\ s_{n+1} \end{pmatrix} &= \begin{pmatrix} \mathbf{y}_n \\ s_n \end{pmatrix} + h_n \begin{pmatrix} \mathbf{g}(s_n, \mathbf{y}_n) \\ 1 \end{pmatrix}, \\
    h_{n+1} &= S(\mathbf{y}_n, h_n, s_n), \label{eq:adaptive-euler-b}
\end{align}
\end{subequations}
where $S$ is a step-size controller, which we assume to be differentiable. The linearization of this method is then
\begin{subequations}\label{eq:adaptive-euler-var}
\begin{align}\label{eq:adaptive-euler-var-a}
    \begin{pmatrix} \delta \mathbf{y}_{n+1} \\ \delta s_{n+1} \end{pmatrix} &= \begin{pmatrix} \delta \mathbf{y}_n \\ \delta s_n \end{pmatrix} + h_n \begin{pmatrix} D_1\mathbf{g}(s_n, \mathbf{y}_n) \delta s_n + D_2\mathbf{g}(s_n,\mathbf{y}_n)\delta \mathbf{y}_n \\ 0 \end{pmatrix} + \delta h_n \begin{pmatrix} \mathbf{g}(s_n, \mathbf{y}_n) \\ 1 \end{pmatrix}, \\
    \delta h_{n+1} &= D_1S(\mathbf{y}_n, h_n, s_n) \delta \mathbf{y}_n + D_2S(\mathbf{y}_n, h_n, s_n)\delta h_n + D_3S(\mathbf{y}_n, h_n, s_n)\delta s_n. \label{eq:adaptive-euler-var-b}
\end{align}
\end{subequations}
Equations \eqref{eq:adaptive-euler} and \eqref{eq:adaptive-euler-var} form the variational system associated with the adaptive Euler method. 
To see that the method is not variationally equivariant, we reverse this process; we first linearize \eqref{eq:time-ind-form-ode} and subsequently, apply the adaptive time-stepping method. The linearization of \eqref{eq:time-ind-form-ode} is 
\begin{align} \label{eq:time-ind-form-variational-ode}
    \frac{d}{dt} \begin{pmatrix} \delta \mathbf{y}(t) \\ \delta s(t) \end{pmatrix} = \begin{pmatrix} D_1\mathbf{g}(s(t), \mathbf{y}(t))\delta s(t) + D_2\mathbf{g}(s(t),y(t))\delta \mathbf{y}(t) \\ 0 \end{pmatrix}.
\end{align}
We then apply the adaptive Euler method to the continuous-time variational system \eqref{eq:time-ind-form-ode}, \eqref{eq:time-ind-form-variational-ode}, which yields
\begin{subequations}\label{eq:adaptive-euler-of-var}
\begin{align} \label{eq:adaptive-euler-of-var-a}
    \begin{pmatrix} \mathbf{y}_{n+1} \\ s_{n+1} \\ \delta \mathbf{y}_{n+1} \\ \delta s_{n+1} \end{pmatrix} &= \begin{pmatrix} \mathbf{y}_n \\ s_n \\ \delta \mathbf{y}_n \\ \delta s_n \end{pmatrix} + h_n \begin{pmatrix} \mathbf{g}(s_n, \mathbf{y}_n) \\ 1 \\ D_1\mathbf{g}(s_n, \mathbf{y}_n) \delta s_n + D_2\mathbf{g}(s_n,\mathbf{y}_n)\delta \mathbf{y}_n \\ 0 \end{pmatrix}, \\
    h_{n+1} &= S(\mathbf{y}_n, h_n, s_n). \label{eq:adaptive-euler-of-var-b}
\end{align}
\end{subequations}
Comparing the variational system of the adaptive method, \eqref{eq:adaptive-euler} and \eqref{eq:adaptive-euler-var}, to the adaptive method applied to the continuous-time variational system, \eqref{eq:adaptive-euler-of-var}, we see that these are not equivalent and hence, the adaptive method is not variationally equivariant. More precisely, the lack of variational equivariance arises from the $\delta h_n$ term in equation \eqref{eq:adaptive-euler-var-a}. We only have variational equivariance when $\delta h_n = 0$, i.e., when $D_1S = 0, D_2S = 0, D_3S = 0$, i.e., there is no adaptive time-stepping. The terms which break the variational equivariance, namely, the derivatives of $S$, are precisely the terms which lead to adjoint inconsistency as shown in \cite{AlSa2009}.
\end{example}

In this section, we considered one-step methods and their associated cotangent lifts. One can ask what happens in the case of multi-step methods. In \cite{Sa2007}, it was shown that the discrete adjoint associated with a multistep method does not generally retain the consistency order of the method used to discretize the state dynamics, both due to error in the backpropagation as well as initialization error for the terminal condition. From our geometric perspective, one-step methods are natural to consider since, at each time-step, a one-step method defines a map $V \rightarrow V$ and hence, one can naturally define its cotangent lift. On the other hand, for a multistep method, at each time step, it defines a map $V \times \dots \times V \rightarrow V$ for which there is no immediately obvious notion of cotangent lift. This can be understood more generally from the fact that a general linear method cannot be symplectic, and thus not a cotangent lift, unless it reduces to a one-step method \cite{BuHe2009}. 

We will now combine the discussions of semi-discretization and time integration together.

\subsection{Naturality of the Full Discretization}\label{section:naturality}

As we have seen, semi-discretization of the evolution equation induces a dual semi-discretization for the corresponding adjoint equation. Furthermore, time integration of the semi-discrete evolution equation induces a dual time integration of the semi-discrete adjoint equation. 

Consider the case of Galerkin semi-discretization. We would like to connect the semi-discretization and time integration commutative diagrams together. To do this, first note, as alluded to in Section \ref{section:semi-discretization}, the process of adjoining at the semi-discrete level depends on the (representation of the) duality pairing. This is of course also true for adjoining at the fully discrete level, since the cotangent lift of a map depends on the duality pairing, equation \eqref{eq:dual-tangent}. 

Combining the discussions of Section \ref{section:semi-discretization} and Section \ref{section:time-integration}, we arrive at the following commutative diagram.
\begin{equation}\label{diagram:full-discrete-adjoint-comm}
\begin{tikzcd}[column sep=14ex, row sep =10ex]
	\dot{q} = g \arrow[rr, "\text{Adjoint}"] \arrow[d, "S"] & {} & \begin{tabular}{c} $\dot{q} = g$ \\ $\dot{p} = -[D_qg]^*p$ \end{tabular} \arrow[d,"(S{,}S^*)"] \\ 
    \dot{\mathbf{q}} = M^{-1}\mathbf{g} \arrow[r, "\text{Adjoint}_S"]  \arrow[rr, "\text{Adjoint}_M", bend left = 20] \arrow[d,"\Phi_{\Delta t}"] & \begin{tabular}{c} $\dot{\mathbf{q}} = M^{-1} \mathbf{g}$ \\ $\dot{\mathbf{z}} = -[D_{\mathbf{q}}\mathbf{g}]^TM^{-T}\mathbf{z}$ \end{tabular} \arrow[d,"(\Phi_{\Delta t}{,}T^{*S}\Phi_{\Delta t}^{-1})"] \arrow[r, "\mathbf{z} = M^T\mathbf{p}", shift left] &  \begin{tabular}{c} $\dot{\mathbf{q}} = M^{-1}\mathbf{g}$ \\ $\dot{\mathbf{p}} = - M^{-T}[D_{\mathbf{q}}\mathbf{g}]^T\mathbf{p}$ \end{tabular} \arrow[l, shift left] \arrow[d,"(\Phi_{\Delta t}{,}T^{*M}\Phi_{\Delta t}^{-1})"] \\
    \mathbf{q}_{n+1} = \Phi_{\Delta t}(\mathbf{q}_n) \arrow[r, "\text{Adjoint}_S"] \arrow[rr, "\text{Adjoint}_M", bend right = 20]& \begin{tabular}{c} $\mathbf{q}_{n+1} = \Phi_{\Delta t}(\mathbf{q}_n)$  \\ $\mathbf{z}_{n+1} = T^{*S}\Phi^{-1}_{\Delta t}(\mathbf{z}_n)$ \end{tabular} \arrow[r, "\text{similarity}", shift left] & \begin{tabular}{c} $\mathbf{q}_{n+1} = \Phi_{\Delta t}(\mathbf{q}_n)$  \\ $\mathbf{p}_{n+1} = T^{*M}\Phi^{-1}_{\Delta t}(\mathbf{p}_n)$ \arrow[l, shift left] \end{tabular} 
\end{tikzcd}
\end{equation}
In the above diagram, we again denote for brevity
\begin{align*}
    g &= Aq + f(t,q), \\
    \mathbf{g} &= K \mathbf{q} + \mathbf{f}(t,\mathbf{q}),
\end{align*}
and have suppressed the additional arguments of the one-step method, but they are given by the corresponding semi-discrete ODE. Other than the arrows denoted ``similarity" in the above diagram, we have already explained all of the elements of the diagram. The equivalence of the two vertices connected by the ``similarity" arrows is the statement that they are related by a similarity transformation $T^{*M}\Phi^{-1}_{\Delta t} = M^{-T}T^{*S}\Phi^{-1}_{\Delta t} M^T$ which immediately follows from the definition of the cotangent lift, equation \eqref{eq:dual-tangent}. 

We also note that by Theorems \ref{thm:semi-disc-galerkin-comm} and \ref{thm:time-int-adjoint-comm}, the semi-discretization and time integration of adjoint systems are uniquely characterized by their respective conservation properties. Thus, to concisely summarize the results derived in this paper, we append these to the diagram.
\begin{equation}\label{diagram:full-discrete-adjoint-comm-cons}
\adjustbox{scale=0.75,center}{
\begin{tikzcd}[column sep=14ex, row sep =10ex] 
	\dot{q} = g \arrow[rr, "\text{Adjoint}"] \arrow[d, "S"] & {} & \begin{tabular}{c} $\dot{q} = g$ \\ $\dot{p} = -[D_qg]^*p$ \end{tabular} \arrow[d,"(S{,}S^*)"] & \arrow[dash, dashed]{l}\frac{d}{dt} \langle p,\delta q\rangle = 0 \\ 
    \dot{\mathbf{q}} = M^{-1}\mathbf{g} \arrow[r, "\text{Adjoint}_S"]  \arrow[rr, "\text{Adjoint}_M", bend left = 20] \arrow[d,"\Phi_{\Delta t}"] & \begin{tabular}{c} $\dot{\mathbf{q}} = M^{-1} \mathbf{g}$ \\ $\dot{\mathbf{z}} = -[D_{\mathbf{q}}\mathbf{g}]^TM^{-T}\mathbf{z}$ \end{tabular} \arrow[d,"(\Phi_{\Delta t}{,}T^{*S}\Phi_{\Delta t}^{-1})"] \arrow[r, "\mathbf{z} = M^T\mathbf{p}", shift left] &  \begin{tabular}{c} $\dot{\mathbf{q}} = M^{-1}\mathbf{g}$ \\ $\dot{\mathbf{p}} = -M^{-T}[D_{\mathbf{q}}\mathbf{g}]^T\mathbf{p}$ \end{tabular} \arrow[l, shift left] \arrow[d,"(\Phi_{\Delta t}{,}T^{*M}\Phi_{\Delta t}^{-1})"] & \arrow[dash, dashed]{l}\frac{d}{dt} \langle \mathbf{p},\delta \mathbf{q}\rangle_M = 0 \\
    \mathbf{q}_{n+1} = \Phi_{\Delta t}(\mathbf{q}_n) \arrow[r, "\text{Adjoint}_S"] \arrow[rr, "\text{Adjoint}_M", bend right = 20]& \begin{tabular}{c} $\mathbf{q}_{n+1} = \Phi_{\Delta t}(\mathbf{q}_n)$  \\ $\mathbf{z}_{n+1} = T^{*S}\Phi^{-1}_{\Delta t}(\mathbf{z}_n)$ \end{tabular} \arrow[r, "\text{similarity}", shift left] &  \arrow[l, shift left] \begin{tabular}{c} $\mathbf{q}_{n+1} = \Phi_{\Delta t}(\mathbf{q}_n)$  \\ $\mathbf{p}_{n+1} = T^{*M}\Phi^{-1}_{\Delta t}(\mathbf{p}_n)$ \end{tabular} & \arrow[dash, dashed]{l} \begin{tabular}{c} $\langle \mathbf{p}_{n+1}, \delta \mathbf{q}_{n+1}\rangle_M$ \\ $= \langle \mathbf{p}_{n}, \delta \mathbf{q}_{n}\rangle_M$ \end{tabular}
\end{tikzcd}. }
\end{equation}

\subsection{Computational Implications} 
We conclude with several remarks on the computational implications of our results. Note that in the diagram \eqref{diagram:full-discrete-adjoint-comm-cons} the arrows marked ``similarity" represent a weaker form of equivalence in the setting of floating point arithmetic. Namely, although they are equivalent assuming exact arithmetic, they will in general not be equivalent with floating point arithmetic; in particular, if $M$ is poorly conditioned or solved indirectly, e.g. using iterative methods as is standard in non-DG FEM discretizations, the two fully discrete methods may produce different results. Also, we mention in Remark \ref{rmk:eq-preservation} that time integration of the semi-discrete adjoint system induced by different duality pairings may lead to different equilibrium characteristics of the discrete solution. Thus, the choice of duality pairing for forming the adjoint system can affect the fully discrete system. 

To elaborate further on the discrepancy between the two vertices related by a similarity transformation, let us focus on the bottom loop of the above diagram and consider a more general method-of-lines semi-discretization, as defined in Section \ref{section:semi-discretization}, and ask generally what is the result of using two different duality pairings. Consider two duality pairings $\langle\cdot,\cdot\rangle_{L}$ and $\langle \cdot,\cdot\rangle_R$ on $X_h \times X_h^*$ which are related by an invertible operator $P$ via
$$ \langle \mathbf{w}, \mathbf{v}\rangle_{L} = \langle \mathbf{w}, P \mathbf{v}\rangle_R. $$
Now, consider a method-of-lines semi-discretization $ \dot{\mathbf{q}} = K_h(\mathbf{q})$ and its time integration by a one-step method $\mathbf{q}_{n+1} = \Phi_{\Delta t}(\mathbf{q})$, where for brevity we denote $K_h = K_h[A,f]$ and $\Phi_{\Delta t}= \Phi_{\Delta t}[n,K_h[A,f]].$
The bottom loop of diagram \eqref{diagram:full-discrete-adjoint-comm} is a special case of the following commutative diagram,
\begin{equation}\label{diagram:preconditioning}
    \begin{tikzcd}[column sep=14ex, row sep =10ex]
        {} & \mathbf{q}_{n+1} = \Phi_{\Delta t}(\mathbf{q}) \arrow[dl, "\text{Adjoint}_L"'] \arrow[dr, "\text{Adjoint}_R"]& {} \\
        \begin{tabular}{c} $\mathbf{q}_{n+1} = \Phi_{\Delta t}(\mathbf{q}_n)$  \\ $\mathbf{z}_{n+1}^L = T^{*L}\Phi^{-1}_{\Delta t}(\mathbf{z}^L_n)$ \end{tabular} \arrow[rr, dashed, shift left, "\text{Similarity}"]  & &  \begin{tabular}{c} $\mathbf{q}_{n+1} = \Phi_{\Delta t}(\mathbf{q}_n)$  \\ $\mathbf{z}^R_{n+1} = T^{*R}\Phi^{-1}_{\Delta t}(\mathbf{z}^R_n)$  \end{tabular} \arrow[ll, dashed, shift left] 
    \end{tikzcd},
\end{equation}
where the similarity transformation is given by $T^{*L}\Phi^{-1}_{\Delta t} = P^{-*R}(T^{*R}\Phi^{-1}_{\Delta t}) P^{*R}$; we have used dashed arrows to emphasize that the similarity transformation is only exact assuming exact arithmetic and may generally produce different results when using floating point arithmetic. Substituting this similarity transformation into the bottom left vertex of diagram \eqref{diagram:preconditioning}, we have
\begin{align*}
    \mathbf{z}_{n+1}^L &= T^{*L}\Phi^{-1}_{\Delta t}(\mathbf{z}^L_n) =P^{-*R}(T^{*R}\Phi^{-1}_{\Delta t}) P^{*R} \mathbf{z}^L_n.
\end{align*}
Equivalently, 
\begin{align}\label{eq:cotangent-precondition}
    P^{*R}\mathbf{z}_{n+1}^L &= T^{*R}\Phi^{-1}_{\Delta t} ( P^{*R} \mathbf{z}^L_n).
\end{align}
Taking $ \mathbf{z}_{n+1}^R = P^{*R}\mathbf{z}_{n+1}^L $ of course reproduces the bottom right vertex of diagram \eqref{diagram:preconditioning}, but note that in general these can produce different numerical results with floating point arithmetic. In particular, if the linear system associated with computing $\mathbf{z}^R_{n+1} = T^{*R}\Phi^{-1}_{\Delta t}(\mathbf{z}^R_n)$ is poorly conditioned, we can instead utilize the formulation \eqref{eq:cotangent-precondition} and view $P^{*R}$ as a preconditioner. 

To elaborate further on this idea of a preconditioner as arising from a choice of duality pairing, we consider the case where the discrete state dynamics $\mathbf{q}_{n+1} = \Phi_{\Delta t}(\mathbf{q}_n)$ is preconditioned by an invertible matrix $P$. Defining $\tilde{\mathbf{q}}_n = P^{-1} \mathbf{q}_n$, the preconditioned discrete state dynamics is given by
$$ P\tilde{\mathbf{q}}_{n+1} = \Phi_{\Delta t}(P\tilde{\mathbf{q}}_n). $$
Equivalently,
$$ \tilde{\mathbf{q}}_{n+1} = (P^{-1} \circ \Phi_{\Delta t} \circ P) (\tilde{\mathbf{q}}_n), $$
which is a one-step method $\tilde{\mathbf{q}}_{n} \mapsto \tilde{\mathbf{q}}_{n+1}$. Computing the cotangent lift of this one-step method with respect to the standard duality pairing $S$, noting that the cotangent lift of a composition of maps is the composition of their cotangent lifts composed in reverse order, we have
$$ \tilde{\mathbf{p}}_{n+1} =  P^T T^{*S}_{P\tilde{\mathbf{q}}_n}\Phi_{\Delta t} P^{-T} \tilde{\mathbf{p}}_n. $$
Equivalently, since $P\tilde{q}_n = q_n$, this can be expressed as
$$ P^{-T} \tilde{\mathbf{p}}_{n+1} = T^{*S}_{\mathbf{q}_n}\Phi_{\Delta t} P^{-T} \tilde{\mathbf{p}}_n. $$
Comparing this to the cotangent lift of the discrete state dynamics without preconditioning, $\mathbf{q}_{n+1} = \Phi_{\Delta t}(\mathbf{q}_n)$, given by $p_{n+1} = T^{*S}_{\mathbf{q}_n}\Phi_{\Delta t} p_n$, we see that they are precisely related by the transformation \eqref{eq:cotangent-precondition}. Thus, the adjoint equation of the preconditioned discrete state dynamics formed with respect to the standard duality pairing is equivalent to the adjoint equation of the discrete state dynamics without preconditioning formed with respect to the duality pairing defined by $\langle \mathbf{w}, \mathbf{v}\rangle_P = \mathbf{w}^T P \mathbf{v}$.  Furthermore, this shows that given a preconditioner $P$ for the discrete state dynamics, the preconditioner for the adjoint equation induced by the cotangent lift is $P^{-T}$. This choice of adjoint preconditioner, given by the inverse transpose of the $P$, can be seen in specific problems in the literature \cite{ScJi2005, HaWa2013}.

\begin{remark}
    In the commutative diagrams presented in this paper, we've included explicitly the standard duality pairing (as this is standard in the literature for forming adjoint systems and similarly in deriving optimility conditions, i.e., the KKT conditions) as well as the mass-matrix induced duality pairing since, as we have seen, it arises naturally in the context of Galerkin semi-discretizations. However, our discussion applies generally to arbitrary choices of duality pairings and, as we have seen, different choices of duality pairings can be put into the practical perspective of preconditioning.
\end{remark}

\begin{remark}[Equilibrium preservation]\label{rmk:eq-preservation}
In non-dissipative PDEs, often one is interested in preservation of equilibria (also known as well-balanced schemes), e.g., \cite{fjordholm2011well}, such that physical equilibria are maintained post spatial and temporal discretization. Consider laminar flow through a domain with constant inflow (Dirichlet) boundary conditions: if initial conditions are in equilibrium (in this case constant and equal to inflow boundary conditions), discretization in space and time should maintain equilibrium for arbitrarily long time. For this example, the spatial discretization would typically have zero row-sums (e.g., consider finite-difference discretizations of advection), wherein the constant vector (laminar flow) is in the nullspace of $M^{-1}K$, and any one-step explicit integration scheme has solution update $\mathbf{g}_{n+1} = \mathbf{g}_n + \Delta tp(M^{-1}K)\mathbf{g}_n = \mathbf{g}_n$, for some polynomial $p(M^{-1}K)$. However, in the standard duality adjoint \eqref{Semi-discrete Adjoint Canonical System}, the resulting scheme is unlikely to preserve equilibria, as $K^TM^{-T}$ will have zero column sums rather than row sums, and constant vectors will no longer be invariant under explicit integration schemes. In contrast, the mass-induced duality pairing \eqref{Semi-discrete Adjoint System} will facilitate equilbiria preservation in this case, as the natural discretization of $K^T$ will maintain the constant nullspace of $K$, invariant to nonuniform scaling by the mass matrix. General equilibria preservation is more complex, but we believe the mass-induced duality pairing is the most natural way to maintain equilbrium preservation in the adjoint equation, while also maintaining discrete symplectic structure of the full adjoint system.
\end{remark}

\section*{Conclusion}
In this paper, we developed a Hamiltonian formulation of the adjoint system associated with an evolutionary partial differential equation. This led to natural geometric characterizations of structure-preserving semi-discretization and time integration in terms of semi-discrete and fully discrete adjoint-variational quadratic conservation laws. In particular, the commutativity of DtO with OtD can be uniquely characterized by these adjoint-variational quadratic conservation laws. 

For future research, we plan to explore applications of this geometric framework in constructing robust geometric discretizations of adjoint systems for evolutionary PDEs. An interesting related direction would be to combine semi-discretization and time integration procedures into a single space-time discretization procedure and analogously examine the question of DtO versus OtD in this more general setting. It is plausible that this question could be analogously characterized in terms of multisymplectic geometry, which is the space-time generalization of symplectic geometry, and multisymplectic discretizations \cite{Ch2005,MaSh1999,BrRe2001,TrLe2022}. Viewing the adjoint system as an equation in space-time instead of as an evolution equation could lead to efficient methods for structure-preserving space-time adjoint sensitivity analysis, utilizing unified space-time discretization methods, e.g.,\cite{hulbert1990space,rhebergen2013space}, and parallel all-at-once space-time solvers, e.g., \cite{horton1995space,danieli2022space}.

\section*{Acknowledgements}
The authors would like to thank the reviewers for their careful review of the paper and for their helpful comments and suggestions. BKT was supported by the Marc Kac Postdoctoral Fellowship at the Center for Nonlinear Studies at Los Alamos National Laboratory. BSS was supported by the Laboratory Directed Research and Development program of Los Alamos National Laboratory under project number 20220174ER. ML was supported in part by NSF under grants DMS-1345013, CCF-2112665, DMS-2307801, by AFOSR under grant FA9550-23-1-0279. Los Alamos National Laboratory report LA-UR-24-22947.

\section*{Data Availability Statement}
Data sharing is not applicable to this article as no datasets were generated or analyzed during the current study.


\bibliographystyle{plainnat}
\bibliography{adjoint_pde.bib}

\end{document}